\documentclass[10pt]{article}
\usepackage{amsfonts}
\usepackage{amsfonts,latexsym,amsmath,amscd,geometry}
\usepackage{amssymb}
\usepackage{latexsym}
\usepackage{amsmath,amssymb,amsthm,amsfonts}
\usepackage{mathrsfs}
\usepackage{bbm}

\usepackage{color}

\newtheorem{lemma}{Lemma}[section]
\newtheorem{theorem}{Theorem}[section]

\newtheorem{proposition}{Proposition}[section]

\numberwithin{equation}{section}

\newcommand{\dis}{\displaystyle}

\newcommand{\R}{\mathbb{R}}

\newcommand{\CA}{\mathcal{A}}

\newcommand{\CE}{\mathcal{E}}

\newcommand{\CI}{\mathcal{I}}
\newcommand{\CJ}{\mathcal{J}}
\newcommand{\CK}{\mathcal{K}}

\newcommand{\CM}{\mathcal{M}}

\newcommand{\CR}{\mathcal{R}}

\newcommand{\al}{\alpha}

\newcommand{\ga}{\gamma}

\newcommand{\de}{\delta}
\newcommand{\si}{\sigma}
\newcommand{\pa}{\partial}
\newcommand{\ka}{\kappa}
\newcommand{\eps}{\epsilon}
\newcommand{\rh}{\rho}
\newcommand{\ta}{\theta}

\newcommand{\eqdef}{\overset{\mbox{\tiny{def}}}{=}}

\begin{document}
\title{{\bf Stability of  contact discontinuity for the Navier-Stokes-Poisson system with free
boundary}}
\author{Shuangqian Liu\thanks{Department of Mathematics, Jinan University, Guangzhou 510632, P.R. China.
Email: tsqliu@jnu.edu.cn},\quad Haiyan Yin\thanks{School of
Mathematical Sciences, Huaqiao University, Quanzhou 362021, P.R.
China. Email: yinhaiyan2000@aliyun.com},\quad Changjiang
Zhu\thanks{Corresponding author. School of Mathematics, South China
University of Technology, Guangzhou 510641, P.R. China. Email:
cjzhu@mail.ccnu.edu.cn}  }

\date{}

%\tableofcontents
%\thispagestyle{empty}

\maketitle

\begin{abstract}
This paper is concerned with the study of the nonlinear stability of the contact discontinuity of
the Navier-Stokes-Poisson system with free boundary in the case where the electron background density satisfies
an analogue of the Boltzmann relation. We especially allow that the electric potential
can take distinct constant states at boundary. On account of the quasineutral assumption,  we first construct
a viscous contact wave through the quasineutral Euler equations, and
then prove that such a non-trivial profile is time-asymptotically stable under small perturbations for the
corresponding initial boundary value problem of the Navier-Stokes-Poisson system.
The analysis is based on the techniques developed in \cite{DL} and an elementary
$L^2$ energy method.
\end{abstract}

\medskip


\medskip


%\maketitle
%\thispagestyle{empty}

%\setcounter{tocdepth}{1}
\tableofcontents

%\newpage
\section{Introduction}
\subsection{The problem}
The dynamics of the charged particles in the collisional dusty plasma %, under the influence of the self-consistent electrostatic potential force
can be described by the Navier-Stokes-Poisson (denoted as NSP in the sequel) system \cite{GSK}. The one-dimensional NSP system in the Eulerian coordinates
takes the form of
\begin{eqnarray}\label{NSPe}
&&\left\{\begin{aligned}
& \pa_{t}\rho+\pa_{x}(\rho u)=0,\\
& \pa_{t}( \rho u)+\pa_{x}(\rho u^{2}+p)
=\rho\pa_{x}\phi+\mu \pa_{x}^{2}u,\\
&\pa_{t}W+\pa_{x}(Wu+pu)=\rho u\pa_{x}\phi+\mu \pa_{x}(u\pa_{x}u)+\kappa\pa_{x}^{2}\ta,\\
&\pa_{x}^{2}\phi=\rho-\rho_{e}(\phi).
\end{aligned}\right.
\end{eqnarray}
The unknown functions $\rho$, $u$ and $\ta$ stand for the
density, velocity and absolute temperature of ions, respectively, while
$\mu>0$ is the viscosity coefficient and $\kappa>0$ is the heat conductivity
coefficient. $W$ stands for the total energy of the
ions, taking the following form:
\begin{eqnarray*}\label{1.2}
W=\frac{\rho u^{2}}{2}+\frac{p}{\gamma-1},
\end{eqnarray*}
where $\gamma>1$ is the adiabatic exponent.  $p$ is the pressure
which is given by
\begin{eqnarray*}\label{1.3}
p=R\rho\ta=A\rho^{\gamma}e^{\frac{\gamma-1}{R}S},
\end{eqnarray*}
where $S$ is the entropy and $A,$ $R$ are both positive constants.
%Note that $\phi$ is scaled as it has opposite sign of the
%electrostatic potential in physics.
The self-consistent electric potential $\phi=\phi(x,t)$ is induced by the total charges through the Poisson equation.
The density $\rho_e=\rho_e(\phi)$ of electrons in \eqref{NSPe} depends only on the potential in the sense of  an analogue of the so-called Boltzmann relation, cf.~\cite{Ch,GP}. Specifically, through the paper we suppose that

\begin{description}
  \item[$(\CA)$] $\rho_e(\phi): (\phi_m,\phi_M)\to (\rho_m,\rho_M)$ is a smooth function with %in $\phi$ in the domain of definition $(\phi_{\rm min},+\infty)$,
 \begin{equation*}
%\label{ }
\rho_m=\inf\limits_{\phi_m<\phi<\phi_M} \rho_e(\phi),\quad \rho_M=\sup\limits_{\phi_m<\phi<\phi_M} \rho_e(\phi),
\end{equation*}
satisfying the following two assumptions:

\medskip
$(\mathcal {A}_1)$ $\rho_e(0)=1$ with $0\in (\phi_m,\phi_M)$;

$(\mathcal {A}_2)$ $\rho_e(\phi)>0$, $\rho'_e(\phi)<0$ for each $\phi\in (\phi_m,\phi_M)$. %for any $\varrho$ in the domain of definition,

%$(\mathcal {A}_3)$ $\rho_e(\phi)\rho''_e(\phi)\leq [\rho'_e(\phi)]^2$ for each $\phi\in (\phi_m,\phi_M)$.

\end{description}
\noindent The assumption $(\mathcal {A}_1)$ just means that the electron density has been normalized to be unit when the potential is zero, since the electric potential in \eqref{NSPe} can be up to an arbitrary constant. The sign of the first derivative of the function $\rh_e(\phi)$ in
the assumption $(\CA_2)$ plays a crucial role in our analysis, it is to be further clarified later on, see \eqref{d.f}, etc.

%The assumption $(\CA_2)$ asserts that $\rho_e(\phi)$ is a positive and decreasing function of $\phi\in (\phi_m,\phi_M)$; it is to be further clarified later on, see \eqref{pr.pp}.
%It should be noted that the above assumptions show that $\rho_e(\phi)$ can be a polynomial or an exponential function such as $e^{A\phi}$ (Boltzmann relation, $A$ is a constant).
%For instance, let us set
An important example satisfying $(\mathcal {A})$ can be given as
\begin{equation}
\label{D-dene}
\rho_e(\phi)=\left[1-\frac{\ga_e-1}{\ga_e} \frac{\phi}{A_e}\right]^{\frac{1}{\ga_e-1}}, \quad \phi_m=-\infty,\quad\phi_{M}=\frac{\ga_e}{\ga_e-1}A_e,
\end{equation}
with $\ga_e\geq 1$ and $A_e>0$ being constants.
%When $\ga_e\to 1+$, $\rho_e(\phi)=e^{\frac{\phi}{A_e}}$.
Note that $\rho_e(\phi)\to e^{-\frac{\phi}{A_e}}$ and $\phi_M\to +\infty$
as $\ga_e\to1^+$, which corresponds to the classical Boltzmann relation. In fact, \eqref{D-dene} can be formally deduced from the momentum equation of the isentropic Euler-Poisson system for the fluid of electrons with the adiabatic exponent $\ga_e$ under the zero-limit of electron mass, namely,
$
\pa_x\left(A_e\rho_e^{\ga_e}\right)=-\rho_e\pa_x\phi.
$

%It should be noted that the above assumptions show that
%$\rho_{e}(\phi)$ can be a polynomial or an exponential function such
%as $e^{-A\phi}$(Boltzmann relation, A is a constant.) For instance,
%let us set
%\begin{eqnarray}\label{as}
%\rho_{e}(\phi)=\left[1-\frac{(\gamma_{e}-1)\phi}{\gamma_{e}A_{e}}\right]^{\frac{1}{\gamma_{e}-1}}
%\end{eqnarray}
%with $\gamma_{e}\geq 1$ and $A_{e}>0$ being constants. When
%$\gamma_{e}\rightarrow 1+,$ $\rho_{e}(\phi)=e^{-A\phi}$. One can see
%that \eqref{as} satisfies the conditions $A_{1}$ and $A_{2}.$ And
%%this special choice corresponds to $\rho_{e}=1$ whenever $\phi=0$.
%This means that the electron density has been normalized to be one
%when there is no potential. Of course $\phi$ can be up to any
%constant. Here, we let $\phi=0$ be the reference potential. In fact,
%\eqref{as} can be solved from
%$$\partial_{x}(A_{e}\rho_{e}^{\gamma_{e}})=-\rho_{e}\partial_{x} \phi.$$
%Substitute $\eqref{1.2}$ and $\eqref{1.3}$ into the equation
% $\eqref{1.1}$, we can obtain the equation as follows:
%\begin{eqnarray}\label{1.4}
%&&\left\{\begin{aligned}
%& \pa_{t}\rho+\pa_{x}(\rho u)=0,\\
%&  \rho (\pa_{t}u+u\pa_{x}u)+\pa_{x}p
%=\rho \pa_{x}\phi+\mu \pa_{x}^{2}u,\\
%&\frac{R}{\gamma-1}\rho\pa_{t}\ta+\frac{R}{\gamma-1}\rho u\pa_{x}\ta+p\pa_{x}u=\mu (\pa_{x}u)^{2}+\kappa\pa_{x}^{2}\ta,\\
%&\pa_{x}^{2}\phi=\rho-\rho_{e}(\phi).
%\end{aligned}\right.
%\end{eqnarray}
%For the cauchy problem, since in \eqref{1.22*x} the viscous contact
%wave can not guarantee $u_{+}=u_{-}$.
In this paper, we consider the system \eqref{NSPe} in the part
$+\infty>x\geq x(t)$, where $x=x(t)$ is a free boundary with the
following dynamical boundary conditions
\begin{equation}\label{BCe}
\frac{dx(t)}{dt}=u(x(t),t),\ \ x(0)=0,\ (p-\mu\partial_{x} u)\mid_{x=x(t)}=p_{-},\ \theta(x(t),t)=\theta_{-},\ \phi(x(t),t)=\phi_{-}.%\ \lim_{x\rightarrow\infty}\phi(x,t)=\phi_{+},
\end{equation}
%and
%\begin{equation}\label{BC2}
%(p-\mu\partial_{x} u)\mid_{x=x(t)}=p_{-},
%\end{equation}
%which means the gas is attached at the free boundary $x=x(t)$ to the
%atmosphere with pressure $p_{-}$ (see\cite{TLOp}).
We also assume $\phi$ satisfies the boundary condition at far field:
\begin{eqnarray}\label{pfc}
\lim_{x\rightarrow+\infty}\phi(x,t)=\phi_{+}.
\end{eqnarray}
The initial
data is given by
\begin{equation}\label{id.e}
(\rho,u,\theta)(x,0)=(\rho_{0},u_{0},\theta_{0})(x),\ \
\lim_{x\rightarrow+\infty}(\rho_{0},u_{0},\theta_{0})(x)=(\rho_{+},u_{+},\theta_{+}).
\end{equation}
Here $\rho_{+}>0$, $\theta_{\pm}>0$, $p_->0$, $u_+$ and $\phi_{\pm}$ are assumed to be constant states. Also, $\rh_0(x)>0$ is supposed, so that the ions flow has no vacuum state.
In addition, we of course assume $\ta_0(x)$ satisfies the compatibility condition
%$\theta_{0}(0)=\theta_{-}$
and $\phi$ satisfies the
the quasineutral condition at far field, i.e.
\begin{equation}\label{cqa.e}
\theta_{0}(0)=\theta_{-},\  \rho_{e}(\phi_{+})=\rho_{+}.
\end{equation}
Our main purpose concerns the large time behavior of solutions to \eqref{NSPe}, \eqref{BCe}, \eqref{pfc} and \eqref{id.e}, to explore this, it is more convenient to
use the Lagrangian coordinates.
%free boundary
%problem become fixed boundary problem which is easier for us to
%study.
That is, consider the coordinate transformation:
\begin{equation*}\label{1.5}
x\Rightarrow\int_{x(t)}^x\rho(y,t)dy,\ \ t\Rightarrow t.
\end{equation*}
We still denote the Lagrangian coordinates by $(x,t)$ for simplicity
of notation. Noticing that
$$
\int_{x(t)}^{x}\rho(y,t)dy\rightarrow+\infty,\ \textrm{as}\ x\rightarrow+\infty,
%\ \textrm{for every fixed}\ t\geq0,
$$
one sees that \eqref{NSPe}, \eqref{BCe}, \eqref{pfc} and \eqref{id.e} can be transformed as the problem with fixed boundary in the form of
\begin{eqnarray}\label{NSPl}
&&\left\{\begin{aligned}
&\pa_{t}v-\pa_{x}u=0,\ \ x>0,\ t>0,\\
& \pa_{t} u+\pa_{x}p
=\frac{\pa_{x}\phi}{v}+\mu \pa_{x}\left(\frac{\pa_{x}u}{v}\right),\ \ x>0,\ t>0,\\
&\frac{R}{\gamma-1}\pa_{t}\ta+p\pa_{x}u=\mu\frac{(\pa_{x}u)^{2}}{v}+\kappa\pa_{x}\left(\frac{\pa_{x}\ta}{v}\right),\ \ x>0,\ t>0,\\
&\pa_{x}\left(\frac{\pa_{x}\phi}{v}\right)=1-v\rho_{e}(\phi),\ \ x>0,\ t>0,\\
\end{aligned}\right.
\end{eqnarray}
with
boundary condition
\begin{equation}\label{BCl}
\theta(0,t)=\theta_{-}, \ \left(p-\mu\frac{\partial_{x}
u}{v}\right)(0,t)=p_{-},\ \phi(0,t)=\phi_{-},\ \lim_{x\rightarrow+\infty}\phi(x,t)=\phi_{+},\ t\geq0,
\end{equation}
and the initial data
\begin{equation}\label{idl}
(v,u,\theta)(x,0)=(v_{0},u_{0},\theta_{0})(x),\ \ \ x\geq0, \ \ \ \
\
\lim_{x\rightarrow+\infty}(v_{0},u_{0},\theta_{0})(x)=(v_{+},u_{+},\theta_{+}).
\end{equation}
Here $v=1/\rho$ stands for the specific volume.  %\eqref{cqa.e}, it follows that
Moreover,
\begin{equation*}\label{cqa}
\ta_0(0)=\ta_-\ \textrm{and}\ v_{+}=\frac{1}{\rho_{e}(\phi_{+})}
\end{equation*}
hold according to \eqref{cqa.e}.

\subsection{Quasineutral Euler equations and contact waves}
In order to study the large time behavior of the solution $[v(x,t),u(x,t),\ta(x,t),\phi(x,t)]$
to the
initial boundary value problem \eqref{NSPl}, \eqref{BCl} and \eqref{idl}, we expect that $[v(x,t),u(x,t),\ta(x,t),\phi(x,t)]$ tends time-asymptotically to viscous contact wave % $[v^{cd}(x,t),u^{cd}(x,t),\ta^{cd}(x,t),\phi^{cd}(x,t)]$,
%where $[v^{cd}(x,t),u^{cd}(x,t),\ta^{cd}(x,t),\phi^{cd}(x,t)]$ is defined to be the viscous contact wave
to the Riemann problem on the
quasineutral Euler system
\begin{eqnarray}\label{MEt}
\left\{
\begin{array}{clll}
\begin{split}
&\pa_tv-\pa_x u=0,\\
&\pa_t u+\pa_xp=\frac{\pa_x\phi}{v},\\
&\frac{R}{\gamma-1}\pa_{t}\theta+p\pa_xu=0,\\
&1/v=\rho_e(\phi),
\end{split}
\end{array}
\right.
\end{eqnarray}
with Riemann initial data given by
\begin{equation}
\label{MEtid}
[v, u, \theta](x,0)= \left\{\begin{array}{rll}[v_-,u_{-}, \theta_-],&\ \ x<0,\\[2mm]
[v_+,u_{+}, \theta_+],&\ \ x>0.
\end{array}
\right.
\end{equation}
According to \cite{Da,S}, one sees that the Riemann problem \eqref{MEt} and \eqref{MEtid}
admits a contact discontinuity solution %(see \cite{sdfg})
\begin{eqnarray}\label{con.dis}
\begin{split}
\left[v^{CD}, u^{CD}, \ta^{CD},\phi^{CD}\right](x,t)=\left\{\begin{array}{rll}[v_-,u_-, \theta_-, \phi_{-}],&\ \ x<0,\\[2mm]
[v_+,u_+, \theta_+, \phi_{+}],&\ \ x>0,
\end{array}
\right.
\end{split}
\end{eqnarray}
on the condition that
\begin{equation}\label{cd.con}
\begin{split}
u_-=u_+,\ \
p_{-}\eqdef p(v_-,\theta_-)=p_{+}+p^{\phi}(v_+)-p^{\phi}(v_-),
\end{split}
\end{equation}
where
\begin{equation*}\label{cd.con2}
p_{+}=p(v_+,\theta_+),\
\phi_{\pm}=\rho_{e}^{-1}(1/v_{\pm}) \ \textrm{and}\ p^{\phi}=p^{\phi}(v)=\int^{v}\frac{1}{\varrho^{3}\rho_{e}'(\rho_{e}^{-1}(\frac{1}{\varrho}))}d\varrho.
\end{equation*}

On the other hand, due to the dissipation effect of the NSP system \eqref{NSPl}, a viscous
contact wave
$\left[v^{cd}, u^{cd}, \ta^{cd},\phi^{cd}\right]$ corresponding to the contact discontinuity
$\left[v^{CD}, u^{CD}, \ta^{CD},\phi^{CD}\right]$ defined as \eqref{con.dis}
can be constructed as follows.
%the influence of viscosity and conductivity of the \eqref{NSPl} is expected to emerge in smoothing effect on contact
%discontinuity
We first denote $p^{cd}=p^{cd}(v^{cd},\ta^{cd})=\frac{R\ta^{cd}}{v^{cd}}$. Since the quasineutral pressure $p^{cd}+p^{\phi}$ for the  profile $\left[v^{cd}, u^{cd}, \ta^{cd},\phi^{cd}\right]$
is expected to be almost constant, we set
\begin{equation}\label{p.cs}
p_{-}=p^{cd}+\int_{v_-}^{v^{cd}}\frac{1}{\varrho^{3}\rho_{e}'(\rho_{e}^{-1}(\frac{1}{\varrho}))}d\varrho.
\end{equation}
Noticing that $\frac{\pa p^{cd}}{\pa v^{cd}}<0$ and $\rho_{e}'(\cdot)<0$, from which and \eqref{p.cs} and the implicit function theorem, we see that
there exists a differentiable function $f(\ta^{cd})$ such that
\begin{equation}\label{r.vta}
v^{cd}=f(\ta^{cd}), \ v_\pm=f(\ta_\pm),
\end{equation}
provided that $|\ta_+-\ta_-|$ is suitably small. Furthermore, by a direct calculation, it follows that
\begin{equation}\label{d.f}
\begin{split}
f'(\theta^{cd})=\frac{R}{p^{cd}-\frac{1}{(v^{cd})^{2}\rho_{e}'(\phi^{cd})}}>0.
\end{split}
\end{equation}
%With \eqref{r.vta} in hand,
We now rewrite the leading part of $\eqref{NSPl}_3$ (the third equation of \eqref{NSPl}) as
\begin{equation}\label{tac.eqn}
\begin{split}
\frac{R}{\gamma-1}\pa_{t}\ta^{cd}+p^{cd}\pa_{x}u^{cd}=\kappa\pa_{x}\left(\frac{\pa_{x}\ta^{cd}}{v^{cd}}\right).
\end{split}
\end{equation}
With \eqref{r.vta} and \eqref{tac.eqn} in hand, we further conjecture that $\left[v^{cd}, u^{cd}, \ta^{cd}\right]$
satisfies
\begin{eqnarray}\label{pro.eqn}
\left\{
\begin{array}{clll}
\begin{split}
&\pa_{t}v^{cd}- \pa_{x}u^{cd}=0,\ v^{cd}=f(\ta^{cd}),\\
&\frac{R}{\gamma-1}\pa_{t}\ta^{cd}+p^{cd}\pa_{x}u^{cd}=\kappa\pa_{x}\left(\frac{\pa_{x}\ta^{cd}}{v^{cd}}\right),\\
%&v^{cd}=f(\ta^{cd}),\ 1/v^{cd}=\rh_e(\phi^{cd}),\\
&\ta^{cd}(0,t)=\theta_{-},\ \ta^{cd}(+\infty,t)=\theta_{+},\ v^{cd}(0,t)=v_-,\ v^{cd}(+\infty,t)=v_+.
\end{split}
\end{array}
\right.
\end{eqnarray}
By virtue of \eqref{pro.eqn}, we obtain a nonlinear diffusion
equation as follows:
\begin{equation}\label{tac.eqn2}
\begin{split}
\pa_{t}\ta^{cd}=\frac{\kappa}{g(\theta^{cd})}\pa_{x}\left(\frac{\pa_{x}\ta^{cd}}{f(\theta^{cd})}\right),\
\ \ta^{cd}(0,t)=\theta_{-},\ \ \ta^{cd}(+\infty,t)=\theta_{+},
\end{split}
\end{equation}
where $g(\theta^{cd})=\frac{R}{\gamma-1}+p^{cd}f'(\theta^{cd})>0.$ Applying the same argument as in \cite{APKo}, one sees that \eqref{tac.eqn2} admits a unique
self similarity solution $\theta^{cd}(\xi)$,
$\xi=\frac{x}{\sqrt{1+t}}$. Additionally, it turns out that $\theta^{cd}$ is
a monotone function, increasing if $\theta_{+}>\theta_{-}$ and
decreasing if $\theta_{+}<\theta_{-}$, and more importantly, one can show that there
exists some positive constant $\overline{\delta}$, such that
for $\delta=|\theta_{+}-\theta_{-}|\leq \overline{\delta}, $
$\theta^{cd}$ satisfies
\begin{equation}\label{tac.pt}
\begin{split}
(1+t)\left|\partial^{2}_{x}\theta^{cd}\right|+(1+t)^{\frac{1}{2}}\left|\partial_{x}\theta^{cd}\right|+\left|\theta^{cd}-\theta_{\pm}\right|
\leq &C\delta e^{-\frac{c_{1}x^{2}}{1+t}}, \ \ \textrm{as}\ \
x\rightarrow +\infty,
\end{split}
\end{equation}
where $c_{1}$ is some positive constant. After $\ta^{cd}$ and $v^{cd}$ are obtained, we now define $\left[u^{cd},\phi^{cd}\right]$
as follows
\begin{eqnarray}\label{pro.eqn2}
\left\{
\begin{array}{clll}
\begin{split}
&\phi^{cd}=\rh_e^{-1}(1/v^{cd}),\\
%&\frac{R}{\gamma-1}\pa_{t}\ta^{cd}+p^{cd}\pa_{x}u^{cd}=\kappa\pa_{x}\left(\frac{\pa_{x}\ta^{cd}}{v^{cd}}\right),\\
&u^{cd}=u_+-\ka\int_x^{+\infty}\frac{f'(\ta^{cd})}{g(\ta^{cd})}\pa_x\left(\frac{\pa_x\ta^{cd}}{f(\ta^{cd})}\right)dx
\\&\quad=u_{+}+\frac{\kappa
f'(\theta^{cd})}{g(\theta^{cd})f(\theta^{cd})}\partial_{x}\theta^{cd}
+\kappa\int_{x}^{+\infty}\frac{(\partial_{x}\theta^{cd})^{2}}{f(\theta^{cd})}\left(\frac{f'}{g}\right)'(\theta^{cd})
dx,\\[2mm]
&\phi^{cd}(0,t)=\phi_-,\ \phi^{cd}(+\infty,t)=\phi_+,\ u^{cd}(+\infty,t)=u_+.
\end{split}
\end{array}
\right.
\end{eqnarray}
It should be noted that $\phi_{\pm}=\rh_e^{-1}(1/v_\pm)$, and $u^{cd}(0,t)$ may not equal to $u_+$.

In view of \eqref{con.dis}, \eqref{pro.eqn}, \eqref{tac.pt} and \eqref{pro.eqn2}, it is straightforward to compute that $\left[v^{cd},u^{cd},\theta^{cd},\phi^{cd}\right]$
satisfies
$$\left\|\left[v^{cd}-v^{CD},u^{cd}-u^{CD},\theta^{cd}-\ta^{CD},\phi^{cd}-\phi^{CD}\right]\right\|_{L^{p}(\R_{+})}
=O\left(\kappa^{\frac{1}{2p}}\right)(1+t)^{\frac{1}{2p}},\ \ p\geq1,$$
which implies the viscous contact wave
$\left[v^{cd},u^{cd},\theta^{cd},\phi^{cd}\right](x,t)$ constructed in \eqref{pro.eqn} and \eqref{pro.eqn2} approximates the contact
discontinuity solution
$\left[v^{CD}, u^{CD}, \ta^{CD},\phi^{CD}\right]$ to the quasineutral Euler
system \eqref{MEt}  in $L^{p}$ norm, $p\geq1$ on any finite time
interval as the heat conductivity coefficients $\kappa$ tends to
zero. %We call $[v^{cd},u^{cd},\theta^{cd}](x,t)$ viscous contact wave.
Moreover, we see that the viscous contact wave
$\left[v^{cd},u^{cd},\theta^{cd},\phi^{cd}\right](x,t)$  solves the Navier-Stokes-Poisson
system \eqref{NSPl} time asymptotically, that is,
\begin{eqnarray*}\label{1.23*x}
\left\{
\begin{array}{clll}
\begin{split}
&\pa_{t}v^{cd}- \pa_{x}u^{cd}=0,\\[2mm]
&\pa_{t} u^{cd}+\pa_{x}p^{cd} =\frac{\pa_{x}\phi^{cd}}{v^{cd}}+\mu
\pa_{x}\left(\frac{\pa_{x}u^{cd}}{v^{cd}}\right)+\CR_{1},\\[2mm]
&\frac{R}{\gamma-1}\pa_{t}\ta^{cd}+p^{cd}\pa_{x}u^{cd}=\mu\frac{(\pa_{x}u^{cd})^{2}}{v^{cd}}
+\kappa\pa_{x}\left(\frac{\pa_{x}\ta^{cd}}{v^{cd}}\right)+\CR_{2},\\[2mm]
&\pa_{x}\left(\frac{\pa_{x}\phi^{cd}}{v^{cd}}\right)=1-v^{cd}\rho_{e}(\phi^{cd})+\CR_{3},
\end{split}
\end{array}
\right.
\end{eqnarray*}
where
\begin{equation*}\label{1.24*x}
\begin{split}
\CR_{1}=&\pa_{t}\left(\frac{\kappa
f'(\theta^{cd})}{g(\theta^{cd})f(\theta^{cd})}\partial_{x}\theta^{cd}
+\int^{+\infty}_{x}\frac{\kappa(\partial_{x}\theta^{cd})^{2}}{f(\theta^{cd})}\left(\frac{f'}{g}\right)'(\theta^{cd})
dx\right)-\mu\partial_{x}\pa_t\left[\ln \left(f(\theta^{cd})\right)\right]\\
=&O(\delta)(1+t)^{-\frac{3}{2}}e^{-\frac{c_{1}x^{2}}{1+t}}, \
\mbox{as}\ \ x\rightarrow +\infty,
\end{split}
\end{equation*}
\begin{equation*}\label{1.25*x}
\begin{split}
\CR_{2}=-\mu\frac{\left(f'(\theta^{cd})\partial_{t}\theta^{cd}\right)^{2}}{f(\theta^{cd})}=O(\delta)(1+t)^{-2}e^{-\frac{c_{1}x^{2}}{1+t}},
\ \mbox{as}\ \ x\rightarrow +\infty,
\end{split}
\end{equation*}
and
\begin{equation*}\label{1.26*x}
\begin{split}
\CR_{3}=\pa_{x}\left(\frac{\pa_{x}\phi^{cd}}{v^{cd}}\right)=O(\delta)(1+t)^{-1}e^{-\frac{c_{1}x^{2}}{1+t}},
\ \mbox{as}\ \ x\rightarrow +\infty.
\end{split}
\end{equation*}

%\begin{equation}\label{vpc.def}
%v^{cd}=f(\ta^{cd}), \ \phi^{cd}=\rh_e^{-1}(1/v^{cd}),
%\end{equation}
%and
%\begin{equation}\label{uc.def}
%\begin{split}
%u^{cd}&=u_++\ka\int_x^{+\infty}\frac{f'(\ta^{cd})}{g(\ta^{cd})}\pa_x\left(\frac{\pa_x\ta^{cd}}{f(\ta^{cd})}\right)dx
%\\&=u_{+}+\frac{\kappa
%f'(\theta^{cd})}{g(\theta^{cd})f(\theta^{cd})}\partial_{x}\theta^{cd}
%+\kappa\int_{x}^{+\infty}\frac{(\partial_{x}\theta^{cd})^{2}}{f(\theta^{cd})}\left(\frac{f'}{g}\right)'(\theta^{cd})
%dx,
%\end{split}
%\end{equation}
%with $v^{cd}(0,t)=f(\ta_-)$, $\lim\limits_{x\rightarrow+\infty}v^{cd}(x,t)=f(\ta_+)$, and $\lim\limits_{x\rightarrow+\infty}u^{cd}(x,t)=u_+$.
\subsection{Main results}
Now we are in a position to state our main results. %For this, let us first define the perturbation
%$$
%\left[\varphi, \psi, \zeta, \sigma\right](x,t)=\left[v-v^{cd},u-u^{cd}, \ta-\ta^{cd},
%\phi-\phi^{cd}\right](x,t).
%$$
%For $T\in(0, +\infty]$ and $\R_+=[0,+\infty]$, we define a function space
%$X([0,T])$ as
%\begin{eqnarray*}
%\begin{aligned}
%X([0,T]):=&\Big\{\left[\varphi, \psi, \zeta, \sigma\right]\Big|[\varphi,\psi]\in C([0,T];H^{1}(\R_+)),\ [\zeta,\sigma]\in C([0,T];H_0^{1}(\R_+)),\
%\\
%&\quad \pa_x\varphi\in L^{2}([0,T];L^{2}(\R_+)),\pa_x(\psi,\zeta,\sigma)\in L^{2}([0,T];H^{1}(\R_+)),\ \ \forall
%(x,t)\in [0,\infty) \times [0,T]\Big\}.
%\end{aligned}
%\end{eqnarray*}
%Concerning the global existence and time-asymptotic properties of
%solutions to the above reformulated initial boundary value problem
%\eqref{3trho}-\eqref{3tid}, one has the following theorem.
%Our main theorem is as follows:
\begin{theorem}\label{main.res.}
For any given $[v_+,u_+,\ta_+, p_-]$ with $v_+>0$ and $\ta_+>0$, suppose that $[v_-,u_-,\ta_-]$ satisfies \eqref{cd.con}, $\phi_\pm=\rh_e^{-1}(v_\pm)$ with $\phi_\pm\in (\phi_m,\phi_M)$, and the function $\rh_e(\cdot)$ satisfies the assumption $(\CA)$.
Let $\left[v^{cd},u^{cd},\ta^{cd},\phi^{cd}\right](x,t)$ be the viscous
contact wave defined in \eqref{pro.eqn} and \eqref{pro.eqn2} with strength
$\delta=|\theta_{+}-\theta_{-}|.$ There
exist positive constants $\epsilon_{0}>0$ and $C_0>0$, such that
if $\left[v_{0}(x)-v^{cd}(x,0),u_{0}(x)-u^{cd}(x,0)\right]\in H^1$, $\left[\ta_{0}(x)-\ta^{cd}(x,0)\right]\in H^1_0$ and
\begin{equation*}\label{p.ID}%+\left\|\left[\phi(x,t)-\phi^{cd}(x,t)\right]\right\|_{H_0^1}
\left\|\left[v_{0}(x)-v^{cd}(x,0),u_{0}(x)-u^{cd}(x,0),\ta_{0}(x)-\ta^{cd}(x,0)\right]\right\|_{H^1}
%+\left\|\left[\ta_{0}(x)-\ta^{cd}(x,0)\right]\right\|_{H_0^1}
+\de\leq
\epsilon_{0},
\end{equation*}
then the initial boundary value problem
\eqref{NSPl}, \eqref{BCl} and \eqref{idl} admits a unique global solution
$[v,u,\ta,\phi](x,t)$ satisfying
$\left[v-v^{cd},u-u^{cd}\right]\in C(0,+\infty; H^1)$, $\left[\ta(x)-\ta^{cd},\phi-\phi^{cd}\right]\in C(0,+\infty; H_0^1)$ and
\begin{equation}\label{main.eng}
\sup\limits_{t\geq0}\left\|\left[v-v^{cd},u-u^{cd},\ta-\ta^{cd},\phi-\phi^{cd}\right]\right\|_{H^1}
%+\left\|\left[\ta-\ta^{cd},\phi-\phi^{cd}\right]\right\|_{H^1_0}
\leq C_0\eps_0^{2/3}.
\end{equation}
Moreover, it holds that
\begin{equation}\label{sol.lag}
\begin{split}
\lim_{t\rightarrow+\infty}\sup_{x\in
\R_+}\left|\left[v-v^{cd},u-u^{cd},\ta-\ta^{cd},\phi-\phi^{cd}\right]\right|=0.
\end{split}
\end{equation}
\end{theorem}
\medskip
From a physical point of view, the motion of the ion-dust plasma (cf.~\cite{KG,GSK}), the self-gravitational viscous gaseous stars (cf.~\cite{Chan}) and the charged particles in semiconductor devices (cf.~\cite{MRC}) can be governed by the NSP system.
%If the effect of Landau damping is neglected, the dynamics of the ion-dust may be governed by the NSP system, cf. \cite{GSK}.
On the other hand,
the NSP system at the fluid level can be justified by taking the hydrodynamical limit of the Vlasov-type Boltzmann equation by the Chapman-Enskog expansion, cf.~\cite{CC,Gr,G,GJ}.
In recent years, there have been a great number of mathematical studies of the NSP system. In what follows, we only mention some of them related to our interest. Ducomet \cite{Du} obtained the existence of nontrivial stationary solutions with compact support and proved the dynamical stability related to a free-boundary value problem for the three-dimensional NSP system in the case that the background profile is vacuum.
 Donatelli \cite{D} established the global existence of weak solutions to the Cauchy problem with large initial data.
Recently, Ding-Wen-Yao-Zhu \cite{DWYZ-09} proved the global existence of weak solutions to the one dimensional isentropic NSP system with density-dependent
viscosity and free boundary.
Donatelli-Marcati \cite{DM} studied the quasineutral limit by using some dispersive estimates of Strichartz type. We point out that some nonexistence result of global weak solutions was also obtained in Chae \cite{Chae}. Zhang-Fang \cite{ZF} studied the  large-time behavior of the spherically symmetric NSP  system with degenerate viscosity coefficients and with vacuum in three dimensions. Jang-Tice \cite{JT} investigated the linear and nonlinear
dynamical instability for the Lane-Emden solutions of the
NSP system in three dimensions under some condition on the adiabatic exponent. Tan-Yang-Zhao-Zou \cite{TYZZ} established the global strong solution to the one-dimensional non-isentropic NSP system with large data for density-dependent viscosity.  In the case when the background profile is strictly positive, the global existence and convergence rates for the three-dimensional NSP system around a non-vacuum constant state were studied by Li-Matsumura-Zhang  \cite{LMZ}, Zhang-Li-Zhu \cite{ZLZ-11} and Hsiao-Li \cite{HL} through carrying out the spectrum analysis.
We point out that Duan \cite{D-NSM} also used the method of Green's function to obtain the large time behaviors of the more complex Navier-Stokes-Maxwell system.

Another interesting and challenging problem is to study the stability of the NSP system on half space, to the best of our knowledge, there are very few results in this line.
Duan-Yang \cite{DY} recently proved the stability of rarefaction wave and boundary layer for outflow problem on the two-fluid NSP system.
The convergence rate of corresponding solutions toward the stationary solution was obtained in Zhou-Li \cite{ZL}. We remark that due to the techniques of the proof, it was assumed in \cite{DY} that all physical parameters in the model must be unit, which is obviously impractical since ions and electrons generally have different masses and temperatures.  One important point used in \cite{DY} is that the large-time behavior of the electric potential is trivial and hence the two fluids indeed have the same asymptotic profiles which are constructed from the Navier-Stokes equations without any force instead of the quasineutral system. Duan-Liu \cite{DL} then improved
the results of \cite{DY} in the sense that all physical constants appearing in the model can be taken in a general way, and the large-time profile of the electric potential is nontrivial on the basis of the quasineutral assumption.
%very recently, when the background profile satisfies the classical Boltzmann relation cf. \cite{GP}, Duan-Liu \cite{DL} proved the nonlinear stability of the rarefaction wave of the isothermal NSP system
%basing on the quasineutral assumption, f
For the investigations in the stability of the rarefaction wave of the related models, see also \cite{DL2} for the study of the more complicated Vlasov-Poisson-Boltzmann system with more general background profile.

When there is no self-consistent force, the NSP system reduces to the well-known Navier-Stokes equations. It is known that there have been extensive investigations
on the stability of wave patterns, namely, shock wave, rarefaction wave, contact discontinuity and their compositions,
in the context of gas dynamical equations and related kinetic equations. Among them, we only mention
\cite{Go,HLM,HXY,JQZ,KT,KZ,LX,LYYZ,LY,M,MN86,MN92,MN-S,PLM,Y} and reference therein.
Moreover, we would also point out some previous works only related to the current work.
Huang-Mastumura-Shi \cite{HMS-04} proved the stability of contact discontinuity of compressible Navier-Stokes equations with free boundary
for the ideal polytropic gas through the construction of viscous contact wave profiles, the key observation in \cite{HMS-04}
is that the asymptotic profile of the temperature $\ta$ satisfies a nonlinear diffusion equation, which can be solved
by the technique developed in \cite{APKo, HLiu}, and later on Huang-Mastumura-Xin \cite{HMX-05} and Huang-Li-Mastumura \cite{HLM}
established the stability of the contact waves of the Cauchy problem.
Recently Huang-Wang-Zhai \cite{HWZ-10} extended the results in \cite{HMS-04} to the general gas, however, for the Cauchy problem,
it still remains an interesting open problem
to generalize the results in \cite{HMX-05, HLM} for the general gas.

In this paper, we intend to study the stability of the contact wave of the NSP system \eqref{NSPe} with free boundary.
Motivated by \cite{DL} and \cite{HMS-04}, we first construct the nontrivial asymptotic profiles of the quasineural Euler equations,
it should be noted that the background density $\rh_e(\phi)$ satisfying assumption $(\CA)$ allows that the asymptotic profile of
the electrical potential can be distinct at the boundary. Then we perform the elementary energy estimates to the perturbative equations
to obtain the global existence and the large time behaviors.
Compared to the classical Navier-Stokes system without any force, the main difficulty in the proof for the NSP system is to treat
the estimates on the terms caused by the potential function $\phi$. Precisely,
the delicate term $\left(\frac{\pa_x\phi}{v}-\frac{\pa_x\phi^{cd}}{v^{cd}}\right)\psi$ can not be directly controlled,
 as in \cite{DL}, the key point to overcome the difficulty is to use the good dissipative property from  the Poisson equation
 by expanding $\rh_e(\phi)$ around the asymptotic profile up to the third-order.
In addition, it is shown \cite{DL} that the sign of the first
derivative of the rarefaction profile of the velocity and the good
time decay properties of the smooth rarefaction profiles are
important to the {\it a priori} estimate. Thus compared with
\cite{DL} in which the stability of the rarefaction wave of the NSP
system is proved, a new difficulty will arise, that is, the critical
term ${\int_{0}^{T}\int_{\R_{+}}}\psi^{2}(\pa_x\theta^{cd})^{2}dxdt$
is beyond control, unlike that of \cite{HLM}, we need to pay extra
effort to take care of the terms involving the self-consistent
force, and it can be seen that the assumption $(\CA_2)$ plays an
essential role to obtain the desired estimates, see Lemma
\ref{key.es} for the details.

\medskip
The rest of the paper is arranged as follows.  In the main part
Section 2, we give the {\it a priori} estimates on the solutions of
the perturbative equations. The proof of Theorem \ref{main.res.} is
concluded in Section 3. In the Appendix, we present the details that
are left in the proofs of the previous sections for completeness of
the paper.

\medskip
\noindent {\it Notations.}  Throughout this paper, we denote a generally large constant by $C$, which may vary from line to line. For two quantities $a$ and
$b$, $a\thicksim b$ means $\frac{1}{C}a \leq  b \leq C a $. $L^p =
L_x^p(\mathbb{\R_{+}}) \ (1 \leqslant p \leqslant \infty)$ denotes
 the usual Lebesgue space on $\R_+=[0,+\infty]$ with its norm $ \|\cdot\|_ {L^p}$, and for convenient, we
write $ \| \cdot \| _{ L^2} = \| \cdot \|$. We also use $H^{k}$ $(k\geq0)$ to denote the usual Sobolev space with respect to $x$ variable on $\R_+$.
%and we write $ \| \cdot \| _{ H^s(\mathbb{\R_{+}})} = \| \cdot
%\|_{s}$.
$C([0,T ]; H^k) (k\geq0)$ denotes the space of the
continuous functions on the interval $[0, T ]$ with values in $H^{k}$.
We use $ (\cdot, \cdot )$ to denote the inner
product over the Hilbert space $ L^{2}$. %i.e.,
%\begin{eqnarray*}
%( f,g )=\int_{\mathbb{\R_{+}}} f(x)g(x)dx,\ \ \ \  f =
%f(x),\ \  g = g(x)\in L^2(\mathbb{\R_{+}}).
%\end{eqnarray*}
$[f_1,f_2]\in H^1$ means $f_1\in H^1$ and $f_2\in H^1$, and so on so forth.

\section{The a priori estimates}

%By, \eqref{1.229*x}, it is easy to verify the following estimates of
%$\theta^{cd}$(see\cite{HMD-04} for details).
%\begin{lemma}\label{cl.Re.Re4.}
%Let $|\theta_{+}-\theta_{-}|=\delta,$ the following estimates hold
%\begin{equation}\label{1.29gsg8*x}
%\begin{split}
%\int_{\R_{+}}(\partial_{x}\theta^{cd})^{4}dx\leq
%C\delta^{4}(1+t)^{-\frac{3}{2}},\ \
%\int_{\R_{+}}(\partial^{2}_{x}\theta^{cd})^{2}dx\leq
%C\delta^{2}(1+t)^{-\frac{3}{2}},
%\end{split}
%\end{equation}
%\begin{equation}\label{1.29g8ko*x}
%\begin{split}
%\int_{\R_{+}}(\partial^{3}_{x}\theta^{cd})^{2}dx\leq
%C\delta^{2}(1+t)^{-\frac{5}{2}},\ \
%\int_{\R_{+}}x\left((\partial_{x}\theta^{cd})^{2}+|\partial^{2}_{x}\theta^{cd}|\right)dx
%\leq C\delta.
%\end{split}
%\end{equation}
%\end{lemma}

%\subsection{The a priori estimates}
 In order to study the stability of contact wave of the initial boundary value problem \eqref{NSPl}, \eqref{BCl} and \eqref{idl}, that is, to prove Theorem \ref{main.res.}, we first
define the perturbation as
$$
[\varphi, \psi, \zeta, \sigma](x,t)=\left[v-v^{cd},u-u^{cd}, \ta-\ta^{cd},
\phi-\phi^{cd}\right](x,t).
$$
Then $[\varphi, \psi, \zeta, \sigma](x,t)$ satisfies
\begin{eqnarray}
&&\pa_t\varphi-\pa_x\psi=0,\label{pv}\\
&&\pa_t \psi+\pa_x
p-\pa_xp^{cd}=\left(\frac{\pa_x\phi}{v}-\frac{\pa_x\phi^{cd}}{v^{cd}}\right)+\mu\pa_x\left(\frac{\pa_x
\psi}{v}\right)+F,\label{pu}\\
&&\frac{R}{\gamma-1}\pa_t\zeta+p\pa_x u-p^{cd}\pa_x u^{cd}=\kappa
\pa_x\left(\frac{\pa_x\ta}{v}-\frac{\pa_x\ta^{cd}}{v^{cd}}\right)+G,\label{pta}\\
&&v^{cd}\pa_x\left(\frac{\pa_x\sigma}{v}\right)=-\varphi+v\left[1-v^{cd}\rho_{e}\left(\sigma+\phi^{cd}\right)\right]
-v^{cd}\pa_x\left(\frac{\pa_x\phi^{cd}}{v}\right),\label{pp}\\
&&\left(p(v,\theta)-\mu\frac{\pa_x u}{v}\right)(0,t)=p_{-},\
\zeta(0,t)=\si(0,t)=\si(+\infty,t)=0,\label{p.BC}\\
&&\left[\varphi,\psi,\zeta](x,0) =[\varphi_0,\psi_0,
\zeta_0\right](x)
\notag\\
&&\qquad\qquad\quad\quad=\left[v_0(x)-v^{cd}(x,0),u_0(x)-u^{cd}(x,0),\ta_0(x)-\ta^{cd}(x,0)\right],\label{p.id}
\end{eqnarray}
where $x\geq0$, $t\geq0$,
%$p^{cd}=p(v^{cd},\ta^{cd})=\frac{R\ta^{cd}}{v^{cd}}$,
$F=-\partial_{t}u^{cd}+\mu\pa_x(\frac{\pa_x u^{cd}}{v})$ and
$G=\mu\frac{(\pa_x u)^{2}}{v}.$ %We note that $\si$ is determined by the elliptic equation \eqref{pp} with the boundary condition that
%$\si(0,t)=\si(+\infty,t)=0$.
We note that the structural identity \eqref{pp} will be of extremal importance for the later proof.

The local existence of \eqref{NSPl}, \eqref{BCl} and \eqref{idl} can
be established by the standard iteration argument cf. \cite{HMS-04}
and hence will be skipped in the paper. To obtain the global
existence part of Theorem \ref{main.res.}, it suffices to prove the
following {\it a priori} estimates. For results in this direction,
we have
\begin{proposition}\label{ape}%(A priori estimates).
Assume all the conditions listed in Theorem \ref{main.res.} hold.
Let $[\varphi,\psi,\zeta,\sigma]$ be a solution to the initial
boundary value problem \eqref{pv}, \eqref{pu}, \eqref{pta}, \eqref{pp}, \eqref{p.BC} and \eqref{p.id} on $0\leq t\leq T$ for some positive
constant T. There are constants $\de>0$, $\eps_0>0$ and $C>0$, such that if $[\varphi,\psi]\in C(0,T; H^1)$, $[\zeta,\sigma]\in C(0,T; H_0^1)$ and
\begin{equation}\label{aps}
\begin{split}
\sup_{0\leq t\leq
T}\left\|[\varphi,\psi,\zeta,\sigma](t)\right\|_{H^1}+\de\leq \eps_{0},
\end{split}
\end{equation}
then the solution  $[\varphi,\psi,\zeta,\sigma](x,t)$ satisfies
\begin{equation}\label{eng.p1}
\begin{split}
\sup_{0\leq t\leq
T}&\left\|[\varphi,\psi,\zeta,\sigma](t)\right\|_{H^1}^2+\int_{0}^{T}\|\pa_x\varphi\|^2
+\|\pa_x\left[\psi,\zeta,\sigma\right]\|_{H^1}^2dt\\
\leq&
C\delta+C\left\|[\varphi_{0},\psi_{0},\zeta_{0}]\right\|^{4/3}_{H^1}.
\end{split}
\end{equation}
\end{proposition}
\begin{proof} We divide it by the following three steps.

%\subsection{The proof of Proposition 2.1.}
 %In this section, we will prove Proposition 2.1 by elementary energy
% method. Now, we prove Proposition 2.1 by the following three steps.\\
\textbf{Step 1.} The zero-order energy estimates.
% Similar to (\cite{HMD-04}),

Multiplying \eqref{pv}, \eqref{pu} and \eqref{pta} by
 $-R\ta^{cd}\left(\frac{1}{v}-\frac{1}{v^{cd}}\right)$, $\psi$ and
$\zeta\theta^{-1}$, respectively, then taking the summation of the resulting
equations, we obtain
\begin{equation}\label{zs}
\begin{split}
\pa_{t}&\left(\frac{1}{2}\psi^{2}+R\theta^{cd}\Phi\left(\frac{v}{v^{cd}}\right)
+\frac{R}{\gamma-1}\theta^{cd}\Phi\left(\frac{\theta}{\theta^{cd}}\right)\right)
+\mu\frac{(\pa_{x}\psi)^{2}}{v}\\&+\frac{\kappa}{v\theta}(\pa_{x}\zeta)^{2}+H_{x}+Q_{1}+Q_{2}=F\psi+\frac{\zeta}{\theta}G
+\underbrace{\left(\frac{\pa_x\phi}{v}-\frac{\pa_x\phi^{cd}}{v^{cd}}\right)\psi}_{I_{1}},
\end{split}
\end{equation}
where
\begin{equation*}\label{4.1}
\begin{split}
\Phi(s)=s-1-\ln s,
\end{split}
\end{equation*}
\begin{equation*}\label{4.2}
\begin{split}
H=\left(p-p^{cd}\right)\psi-\mu\frac{\psi\pa_x\psi}{v}-\kappa\frac{\zeta}{\theta}\left(\frac{\pa_x\theta}{v}-\frac{\pa_x\theta^{cd}}{v^{cd}}\right),
\end{split}
\end{equation*}
\begin{equation*}\label{4.3}
\begin{split}
Q_{1}=-R\pa_t\theta^{cd}\Phi\left(\frac{v}{v^{cd}}\right)-p^{cd}\pa_tv^{cd}\left(2-\frac{v}{v^{cd}}-\frac{v^{cd}}{v}\right)
+\frac{R}{\gamma-1}\pa_t\theta^{cd}
\Phi\left(\frac{\theta^{cd}}{\theta}\right)+\frac{\zeta}{\theta}\left(p-p^{cd}\right)\pa_xu^{cd},
\end{split}
\end{equation*}
and
\begin{equation*}\label{4.4}
\begin{split}
Q_{2}=-\kappa\frac{\pa_x\theta}{\theta^{2}v}\zeta\pa_x\zeta-\kappa\frac{\varphi\pa_x\zeta}{\theta
vv^{cd}}\pa_x\theta^{cd}+\kappa\frac{\zeta\varphi\pa_x\theta}{\theta^{2}
vv^{cd}}\pa_x\theta^{cd}.
\end{split}
\end{equation*}
%Furthermore, we have
%\begin{equation}\label{4.5}
%\begin{split}
%Q_{1}=p^{cd}\partial_{x}u^{cd}\Phi(\frac{
%v^{cd}}{v})+\left(p^{cd}-\frac{R}{f'(\theta^{cd})}\right)\partial_{x}u^{cd}\Phi(\frac{
%v}{v^{cd}})+\frac{R}{\gamma-1}\pa_t\theta^{cd}
%\Phi(\frac{\theta^{cd}}{\theta})+\frac{\zeta}{\theta}(p-p^{cd})\pa_xu^{cd}
%\end{split}
%\end{equation}
%by virtue of $v^{cd}=f(\theta^{cd})$ and $\pa_tv^{cd}=\pa_xu^{cd}.$
%Moreover, one can see that
%\begin{equation}\label{4S_{1}}
%\begin{split}
%\pa_xv^{cd}=f'(\theta^{cd})\pa_x\theta^{cd}
%\end{split}
%\end{equation}
%and
%\begin{equation}\label{4S_{2}}
%\begin{split}
%\pa_xu^{cd}=f'(\theta^{cd})\pa_t\theta^{cd}.
%\end{split}
%\end{equation}

Let us now consider the most delicate term $I_1$ on the right hand
side of \eqref{zs}. The key technique to handle $I_1$ is to use the
good dissipative property of the Poisson equation by expanding
$\rho_{e}(\sigma+\phi^{cd})$ around the asymptotic profile up to the
third-order. Only in this way, we can observe some new cancelations
and obtain the higher order nonlinear terms.

With the aid of \eqref{pp} and
\eqref{pv},  one has
\begin{equation}\label{I1}
\begin{split}
I_1=&-\frac{\pa_x\psi\sigma}{v}
+\frac{\psi\pa_x\varphi\sigma}{v^2}+\frac{\psi\pa_xv^{cd}\sigma}{v^2}
+\frac{\psi\pa_xv^{cd}\varphi}{(v^{cd})^3v\rho_{e}'(\phi^{cd})}+\pa_x\left(\frac{\sigma\psi}{v}\right)\\
=&\underbrace{-\pa_t\left[-v^{cd}\pa_x\left(\frac{\pa_x\sigma}{v}\right)+v\left(1-v^{cd}\rho_{e}(\sigma+\phi^{cd})\right)
-v^{cd}\pa_x\left(\frac{\pa_x\phi^{cd}}{v}\right)\right]\sigma
v^{-1}}_{I_{1,1}}
\\
&\underbrace{+\pa_x\left[-v^{cd}\pa_x\left(\frac{\pa_x\sigma}{v}\right)+v\left(1-v^{cd}\rho_{e}(\sigma+\phi^{cd})\right)
-v^{cd}\pa_x\left(\frac{\pa_x\phi^{cd}}{v}\right)\right]\psi\sigma v^{-2}}_{I_{1,2}}\\
&\underbrace{+\left[-v^{cd}\pa_x\left(\frac{\pa_x\sigma}{v}\right)+v\left(1-v^{cd}\rho_{e}(\sigma+\phi^{cd})\right)
-v^{cd}\pa_x\left(\frac{\pa_x\phi^{cd}}{v}\right)\right]\psi\pa_xv^{cd}(v^{cd})^{-3}[v\rho_{e}'(\phi^{cd})]^{-1}}_{I_{1,3}}\\
&\underbrace{+\frac{\psi\pa_xv^{cd}\sigma}{v^2}}_{I_2}+\pa_x\left(\frac{\sigma\psi}{v}\right).
\end{split}
\end{equation}
To deal with the lower order terms involving $1-v^{cd}\rho_{e}(\sigma+\phi^{cd})$,
we first get from the Taylor's formula with an integral remainder
that
\begin{equation}\label{taylor}
1-v^{cd}\rho_{e}(\sigma+\phi^{cd})=-v^{cd}\rho_{e}'(\phi^{cd})\sigma-\frac{v^{cd}\rho_{e}''(\phi^{cd})}{2}\sigma^{2}
\underbrace{-v^{cd}\int_{\phi^{cd}}^{\phi}\rho_{e}'''(\varrho)\frac{(\phi-\varrho)^{2}}{2}d\varrho}_{I_0}.
\end{equation}
By virtue of \eqref{taylor}, we then compute $I_{1,1}$, $I_{1,2}$ and $I_{1,3}$ as follows:
\begin{equation}\label{I1.1}
\begin{split}
I_{1,1}=&-\frac{1}{2}\pa_t\left(\frac{v^{cd}}{v^2}(\pa_x\sigma)^2\right)
-\pa_t\left(\frac{\pa_x\sigma}{v}\sigma\pa_x\left(\frac{v^{cd}}{v}\right)\right)
-\frac{1}{2}\pa_t\left(\frac{v^{cd}}{v^2}\right)(\pa_x\sigma)^2\\
&+\pa_x\left(\frac{v^{cd}}{v}\right)\frac{\pa_x\sigma}{v}\pa_t\sigma
-v^{cd}\pa_x\left(\frac{\pa_x\sigma}{v}\right)\sigma\pa_t(v^{-1})
+\frac{1}{2}\pa_t\left(v^{cd}\rho_{e}'(\phi^{cd})\sigma^2\right)
\\&+\frac{1}{3}\pa_t\left(v^{cd}\rho_{e}''(\phi^{cd})\sigma^3\right)-
\pa_t I_0\sigma
\underbrace{+\frac{\pa_tv}{v}v^{cd}\rho_{e}'(\phi^{cd})\sigma^2}_{I_3}
+\frac{1}{2}\pa_{t}\left(v^{cd}\rho_{e}'(\phi^{cd})\right)\sigma^2
\\&+\frac{1}{2}\frac{\pa_tv}{v}v^{cd}\rho_{e}''(\phi^{cd})\sigma^3
+\frac{1}{6}\pa_{t}\left(v^{cd}\rho_{e}''(\phi^{cd})\right)\sigma^3
-\frac{\pa_tv}{v}I_0\sigma
+\pa_t\left(v^{cd}\pa_x\left(\frac{\pa_x\phi^{cd}}{v}\right)\right)\sigma
v^{-1}\\&+\pa_x\pa_t\left(\frac{v^{cd}\sigma\pa_x\sigma}{v^{2}}\right)-\pa_x\left(\frac{v^{cd}\pa_t\sigma\pa_x\sigma}{v^{2}}\right),
\end{split}
\end{equation}
\begin{equation}\label{I1.2}
\begin{split}
I_{1,2}=&\frac{v^{cd}}{v^3}\pa^2_x\sigma\pa_x\sigma\psi
+\frac{v^{cd}}{v}\pa^2_x\sigma \sigma \pa_x(\psi v^{-2})
+v^{cd}\pa_x\sigma \pa_x(v^{-1})\pa_x(\sigma \psi v^{-2})\\
&\underbrace{-v^{cd}\rho_{e}'(\phi^{cd})\frac{\pa_x\sigma
\sigma\psi}{v}-
v^{cd}\rho_{e}'(\phi^{cd})\frac{\pa_xv\sigma^2\psi}{v^2}-\pa_{x}\left(v^{cd}\rho_{e}'(\phi^{cd})\right)\frac{\sigma^2\psi}{v}}_{I_4}
-v^{cd}\rho_{e}''(\phi^{cd})\frac{\pa_x\sigma\sigma^2\psi}{v}
\\&-\frac{1}{2}v^{cd}\rho_{e}''(\phi^{cd})\frac{\pa_xv\sigma^3\psi}{v^2}
-\frac{1}{2}\pa_{x}\left(v^{cd}\rho_{e}''(\phi^{cd})\right)\frac{\sigma^3\psi}{v}
+\frac{\pa_xI_0\sigma\psi}{v} +\frac{\pa_xvI_0\sigma\psi}{v^2}\\
& +v^{cd}\pa_x\left(\frac{\pa_x\phi^{cd}}{v}\right)\pa_x\left(\psi\sigma
v^{-2}\right)-\pa_x\left(v^{cd}\pa_x\left(\frac{\pa_x\phi}{v}\right)\psi\sigma
v^{-2} \right),
\end{split}
\end{equation}
\begin{equation}\label{I1.3}
\begin{split}
I_{1,3}=&-\pa_x\left(\frac{\pa_x\sigma}{v}\right)
\frac{\psi\pa_xv^{cd}}{v(v^{cd})^{2}\rho_{e}'(\phi^{cd})}
\underbrace{-\frac{\pa_xv^{cd}\sigma\psi}{(v^{cd})^2}}_{I_5}
\underbrace{-\frac{1}{2}\frac{\rho_{e}''(\phi^{cd})\pa_xv^{cd}\sigma^2\psi}{(v^{cd})^2\rho_{e}'(\phi^{cd})}}_{I_6}\\
&+\frac{\pa_xv^rI_0\psi}{(v^{cd})^3\rho_{e}'(\phi^{cd})}
-\pa_x\left(\frac{\pa_x\phi^{cd}}{v}\right)\frac{\psi\pa_xv^{cd}}{v(v^{cd})^2\rho_{e}'(\phi^{cd})}.
\end{split}
\end{equation}
%It is interesting to remark that
%We find that some linear terms will cancel out when putting $I_l$ $(2\leq l\leq 7)$ together.
%together, one sees that some linear terms will vanish.
Note that $I_l$ $(2\leq l\leq 6)$ can not be directly controlled.
To overcome this difficulty, we first get from \eqref{pp} and \eqref{taylor} that
\begin{equation*}
\begin{split}
(I_2+I_5)+I_6+I_4=&-\frac{\pa_xv^{cd}\sigma\psi\varphi(v+v^{cd})}{(vv^{cd})^2}
-\frac{\rho_{e}''(\phi^{cd})\pa_xv^{cd}\sigma^2\psi}{2\rho_{e}'(\phi^{cd})(v^{cd})^2}
-v^{cd}\rho_{e}'(\phi^{cd})\frac{\pa_x\sigma
\sigma\psi}{v}\\&-\frac{\pa_xv\sigma^2\psi}{v^2}v^{cd}\rho_{e}'(\phi^{cd})
-\left(v^{cd}\rho_{e}'(\phi^{cd})\right)_{x}\frac{\sigma^2\psi}{v}
\\=&\frac{\pa_xv^{cd}\sigma^2\psi}{(v^{cd})^2}\left[v^{cd}\rho_{e}'(\phi^{cd})-\frac{\rho_{e}''(\phi^{cd})}{2\rho_{e}'(\phi^{cd})}\right]
-\frac{\pa_x\sigma\sigma\psi}{v}v^{cd}\rho_{e}'(\phi^{cd})
+\frac{\pa_xv^{cd}\sigma^2\psi}{v} \frac{\rho_{e}''(\phi^{cd})}{v^{cd}\rho_{e}'(\phi^{cd})}\\
&-\frac{\pa_xv\sigma^2\psi}{v^2}v^{cd}\rho_{e}'(\phi^{cd})+\frac{1}{2}\frac{\pa_xv^{cd}\sigma^3\psi(v+v^{cd})}{vv^{cd}}\rho_{e}''(\phi^{cd})
-\frac{\pa_xv^{cd}\sigma\psi I_0(v+v^{cd})}{v(v^{cd})^2}\\
&+\frac{\pa_xv^{cd}\sigma\psi\pa_x\left(\frac{\pa_x\phi^{cd}}{v}\right)(v+v^{cd})}{v^2v^{cd}}
+\frac{\pa_xv^{cd}\sigma\psi\pa_x\left(\frac{\pa_x\sigma}{v}\right)(v+v^{cd})}{v^2v^{cd}},
\end{split}
\end{equation*}
which is further equal to
\begin{equation}\label{sumJ}
\begin{split}
&\underbrace{\frac{1}{2}\frac{\pa_x\psi\sigma^2}{v}v^{cd}\rho_{e}'(\phi^{cd})}_{I_7}
+\sigma^2\psi
v^{cd}\rho_{e}'(\phi^{cd})\pa_x\left(\frac{1}{v}-\frac{1}{v^{cd}}\right)+\frac{1}{2}\frac{\sigma^2\psi\varphi}{v^{2}}\pa_xv^{cd}
\left[\rho_{e}'(\phi^{cd})-\frac{v\rho_{e}(\phi^{cd})\rho_{e}''(\phi^{cd})}{v^{cd}\rho_{e}'(\phi^{cd})}\right]
\\
&-\frac{1}{2}\frac{\sigma^2\psi\pa_x\varphi}{v^{2}}v^{cd}\rho_{e}'(\phi^{cd})
+\frac{1}{2}\frac{\pa_xv^{cd}\sigma^3\psi(v+v^{cd})}{vv^{cd}}\rho_{e}''(\phi^{cd})
-\frac{\pa_xv^{cd}\sigma\psi I_0(v+v^{cd})}{v(v^{cd})^2}\\
&+\frac{\pa_xv^{cd}\sigma\psi\pa_x\left(\frac{\pa_x\phi^{cd}}{v}\right)(v+v^{cd})}{v^2v^{cd}}
+\frac{\pa_xv^{cd}\sigma\psi\pa_x\left(\frac{\pa_x\sigma}{v}\right)(v+v^{cd})}{v^2v^{cd}}-\pa_x\left(\frac{\psi\sigma^2}{2v}v^{cd}\rho_{e}'(\phi^{cd})\right).
\end{split}
\end{equation}
For $I_3$ and $I_7$,  it follows from $\eqref{NSPl}_1$, \eqref{pv} and
\eqref{pp} that
\begin{equation}\label{J3pJ7}
\begin{split}
I_3+I_7=&\frac{3\pa_t\varphi\sigma^2}{2v}v^{cd}\rho_{e}'(\phi^{cd})
+\frac{\pa_xu^{cd}\sigma^2}{v}v^{cd}\rho_{e}'(\phi^{cd})\\
=&\pa_t\left(\frac{3\varphi\sigma^2}{2v}v^{cd}\rho_{e}'(\phi^{cd})\right)
-\frac{3}{2}v^{cd}\rho_{e}'(\phi^{cd})\varphi\pa_{t}\left(\frac{\sigma^2}{v}\right)\\
&-\frac{3}{2}\pa_{t}\left(v^{cd}\rho_{e}'(\phi^{cd})\right)\frac{\varphi\sigma^2}{v}
+\frac{\pa_xu^{cd}\sigma^2}{v}v^{cd}\rho_{e}'(\phi^{cd}).
\end{split}
\end{equation}
Plugging \eqref{J3pJ7}, \eqref{sumJ}, \eqref{I1.3}, \eqref{I1.2},
\eqref{I1.1} and \eqref{I1} into \eqref{zs}, integrating the resulting identity with respect to $x$ over $\R_+$, and using $(\mathcal
{A}_{2})$,  we thus arrive at
\begin{equation}\label{zs2}
\begin{split}
\frac{d}{dt}&\int_{\R_{+}}\left(\frac{1}{2}\psi^{2}+R\theta^{cd}\Phi\left(\frac{v}{v^{cd}}\right)
+\frac{R}{\gamma-1}\theta^{cd}\Phi\left(\frac{\theta}{\theta^{cd}}\right)
+\frac{v^{cd}}{2}|\rho_{e}'(\phi^{cd})|\sigma^2+\frac{v^{cd}}{2v^2}(\pa_x\sigma)^2\right)dx
\\&+\frac{d}{dt}\int_{\R_{+}}\frac{\pa_x\sigma}{v}\sigma\pa_x\left(\frac{v^{cd}}{v}\right)dx
-\frac{3}{2}\frac{d}{dt}\int_{\R_{+}}v^{cd}\rho_{e}'(\phi^{cd})\frac{\varphi\sigma^2}{v}dx
-\frac{1}{3}\frac{d}{dt}\int_{\R_{+}}v^{cd}\rho_{e}''(\phi^{cd})\sigma^3dx
\\&+\mu\int_{\R_{+}}\frac{(\pa_{x}\psi)^{2}}{v}dx+\int_{\R_{+}}\frac{\kappa}{v\theta}(\pa_{x}\zeta)^{2}dx
\\=&-\int_{\R_{+}}Q_{1}dx
-\int_{\R_{+}}Q_{2}dx+\int_{\R_{+}}F\psi dx
+\int_{\R_{+}}\frac{\zeta}{\theta}G dx
+\widetilde{H}(0,t)+\sum\limits_{l=1}^{31}\CI_l,
\end{split}
\end{equation}
where
\begin{equation*}\label{78uiH}
\begin{split}
\widetilde{H}=\widetilde{H}(x,t)=&(p-p^{cd})\psi-\mu\frac{\psi\pa_x\psi}{v}
-\kappa\frac{\zeta}{\theta}\left(\frac{\pa_x\theta}{v}-\frac{\pa_x\theta^{cd}}{v^{cd}}\right)
-\frac{\sigma\psi}{v}-\pa_t\left(\frac{v^{cd}\sigma\pa_x\sigma}{v^{2}}\right)\\&+\frac{v^{cd}\pa_t\sigma\pa_x\sigma}{v^{2}}
+v^{cd}\pa_x\left(\frac{\pa_x\phi}{v}\right)\psi\sigma v^{-2}
+\frac{\psi\sigma^2}{2v}v^{cd}\rho_{e}'(\phi^{cd}),
\end{split}
\end{equation*}
and
\begin{eqnarray*}
\left\{\begin{array}{rll}
\begin{split}
&\CI_1=\int_{\R_{+}}\pa_x\left(\frac{v^{cd}}{v}\right)\frac{\pa_x\sigma}{v}\pa_t\sigma
dx,\ \ \
\CI_2=-\int_{\R_{+}}v^{cd}\pa_x\left(\frac{\pa_x\sigma}{v}\right)\sigma\pa_t(v^{-1})dx,\\
&\CI_3=-\int_{\R_{+}}\pa_t I_0\sigma dx,\ \ \
\CI_4=-\frac{1}{2}\int_{\R_{+}}\pa_t\left(\frac{v^{cd}}{v^2}\right)(\pa_x\sigma)^2dx,\\
&\CI_5=-\int_{\R_{+}}\frac{\pa_tv}{v}I_0\sigma dx,\ \ \
\CI_6=\int_{\R_{+}}\pa_t\left(v^{cd}\pa_x\left(\frac{\pa_x\phi^{cd}}{v}\right)\right)\sigma
v^{-1}dx,
\\
&\CI_7=\int_{\R_{+}}\frac{v^{cd}}{v^3}\pa^2_x\sigma\pa_x\sigma\psi dx,\ \ \
\CI_8=\int_{\R_{+}}\frac{v^{cd}}{v}\pa^2_x\sigma\sigma\pa_x(\psi v^{-2})dx,\\
&\CI_{9}=\int_{\R_{+}}v^{cd}\pa_x\sigma\pa_x(v^{-1})\pa_x(\sigma\psi
v^{-2})dx,\ \ \
\CI_{10}=\frac{1}{2}
\int_{\R_{+}}\pa_t(v^{cd}\rho_{e}'(\phi^{cd}))\sigma^2dx,\\
&\CI_{11}=\frac{1}{2}\int_{\R_{+}}\frac{\pa_tv}{v}v^{cd}\rho_{e}''(\phi^{cd})\sigma^3dx,\ \ \
\CI_{12}=\int_{\R_{+}}\frac{\pa_xv^{cd}I_0\psi}{(v^{cd})^3\rho_{e}'(\phi^{cd})}dx,\\
&\CI_{13}=\int_{\R_{+}}\frac{\pa_x v I_0\sigma\psi}{v^2}dx,\ \ \
\CI_{14}=\int_{\R_{+}}\frac{\pa_xI_0\sigma\psi}{v}dx,\\
&\CI_{15}=-\int_{\R_{+}}\pa_x\left(\frac{\pa_x\sigma}{v}\right)
\frac{\psi\pa_xv^{cd}}{v(v^{cd})^{2}\rho_{e}'(\phi^{cd})}dx,\ \ \
\CI_{16}=\int_{\R_{+}}v^{cd}\pa_x\left(\frac{\pa_x\phi^{cd}}{v}\right)\pa_x\left(\psi\sigma
v^{-2}\right)dx,\\
\end{split}
\end{array}\right.
\end{eqnarray*}
\begin{eqnarray*}
\left\{\begin{array}{rll}
\begin{split}
&\CI_{17}=-\int_{\R_{+}}\pa_x\left(\frac{\pa_x\phi^{cd}}{v}\right)\frac{\psi\pa_xv^{cd}}{v(v^{cd})^2\rho_{e}'(\phi^{cd})}
dx,\ \ \
\CI_{18}=\frac{1}{6}\int_{\R_{+}}\pa_{t}\left(v^{cd}\rho_{e}''(\phi^{cd})\right)\sigma^3dx,\\
&\CI_{19}=-\frac{1}{2}\int_{\R_{+}}v^{cd}\rho_{e}''(\phi^{cd})\frac{\pa_xv\sigma^3\psi}{v^2}dx,\
\ \ \
\CI_{20}=-\frac{1}{2}\int_{\R_{+}}\pa_{x}\left(v^{cd}\rho_{e}''(\phi^{cd})\right)\frac{\sigma^3\psi}{v}dx,\\
&\CI_{21}=-\int_{\R_{+}}v^{cd}\rho_{e}''(\phi^{cd})\frac{\pa_x\sigma\sigma^2\psi}{v}dx,\
\ \ \
\CI_{22}=\int_{\R_{+}}\sigma^2\psi v^{cd}\rho_{e}'(\phi^{cd})\pa_x\left(\frac{1}{v}-\frac{1}{v^{cd}}\right)dx,\\
&\CI_{23}=\frac{1}{2}\int_{\R_{+}}\frac{\sigma^2\psi\varphi}{v^{2}}\pa_xv^{cd}
\left[\rho_{e}'(\phi^{cd})-\frac{v\rho_{e}(\phi^{cd})\rho_{e}''(\phi^{cd})}{v^{cd}\rho_{e}'(\phi^{cd})}\right]dx,\\
&\CI_{24}=-\frac{1}{2}\int_{\R_{+}}\frac{\sigma^2\psi\pa_x\varphi}{v^{2}}v^{cd}\rho_{e}'(\phi^{cd})dx,\
\ \ \
\CI_{25}=\frac{1}{2}\int_{\R_{+}}\frac{\pa_xv^{cd}\sigma^3\psi(v+v^{cd})}{vv^{cd}}\rho_{e}''(\phi^{cd})dx,\\
&\CI_{26}=-\int_{\R_{+}}\frac{\pa_xv^{cd}\sigma\psi
I_0(v+v^{cd})}{v(v^{cd})^2}dx,\ \ \
\CI_{27}=\int_{\R_{+}}\frac{\pa_xv^{cd}\sigma\psi\pa_x\left(\frac{\pa_x\phi^{cd}}{v}\right)(v+v^{cd})}{v^2v^{cd}}dx,\\
&
\CI_{28}=\int_{\R_{+}}\frac{\pa_xv^{cd}\sigma\psi\pa_x\left(\frac{\pa_x\sigma}{v}\right)(v+v^{cd})}{v^2v^{cd}}dx,\
\ \ \
\CI_{29}=\int_{\R_{+}}\frac{\pa_xu^{cd}\sigma^2}{v}v^{cd}\rho_{e}'(\phi^{cd})dx,\\
&\CI_{30}=-\frac{3}{2}\int_{\R_{+}}\pa_{t}\left(v^{cd}\rho_{e}'(\phi^{cd})\right)\frac{\varphi\sigma^2}{v}dx,\
\ \
\CI_{31}=-\frac{3}{2}\int_{\R_{+}}v^{cd}\rho_{e}'(\phi^{cd})\varphi\pa_{t}\left(\frac{\sigma^2}{v}\right)dx.
\end{split}
\end{array}\right.
\end{eqnarray*}
We now turn to estimate the right hand side of \eqref{zs2} term by term.
It should be noted that the following Poincar\'{e} type inequalities play an important role in our computations:
\begin{equation}\label{p.ine}
\begin{split}
|\zeta(x,t)|\leq x^{\frac{1}{2}}\|\pa_x\zeta\|,\ \
|\varphi(x,t)|\leq |\varphi(0,t)|+x^{\frac{1}{2}}\|\pa_x\varphi\|, \
|\sigma(x,t)|\leq x^{\frac{1}{2}}\|\pa_x\sigma\|.
\end{split}
\end{equation}
From \eqref{p.ine} and Lemma \ref{es.tap},
one can further obtain
\begin{eqnarray}\label{p.ine2}
\left\{\begin{array}{rll}
&\dis{\int_{\R_{+}}}\varphi^{2}\left((\partial_{x}\theta^{cd})^{2}+|\partial^{2}_{x}\theta^{cd}|\right)dx
\leq
C\delta^2\|\varphi_{0}\|_{H^1}^{2}e^{-\frac{p_{-}}{\mu}t}+C\delta^2\|\partial_{x}\varphi\|^{2},\\[4mm]
&\dis{\int_{\R_{+}}}(\zeta^{2}+\sigma^{2})\left((\partial_{x}\theta^{cd})^{2}+|\partial^{2}_{x}\theta^{cd}|\right)dx
\leq
C\delta^2\|\partial_{x}[\zeta,\sigma]\|^{2},
\end{array}\right.
\end{eqnarray}
where the following Sobolev inequality is also used:
\begin{equation}\label{sob.}
%|h(x)|_{L^\infty}\leq C\|h(x)\|_{H^1},\ \textrm{and}\
|h(x)|\leq\sqrt{2}\|h\|^{1/2}\|\pa_xh\|^{1/2}\ \textrm{for}\ h(x)\in H^1({\R_+}).
\end{equation}
%We now present the calculations for terms involving $Q_1$, $Q_2$, $F$ and $G$ as follows.
%Applying the Nirenberg inequlity, the Young inequality,
%\eqref{1.229*x}, \eqref{1.22*x}, \eqref{4.5} and \eqref{retg}, one
%concludes that
By applying \eqref{p.ine2}, Lemma \ref{pbd}, the {\it a priori}
assumption \eqref{aps}, Cauchy-Schwarz's inequality with $0<\eta<1$
and Sobolev's inequality \eqref{sob.}, we obtain the estimates for
terms involving $Q_1$ and $Q_2$ as follows:
\begin{equation}\label{Q1}
\begin{split}
\left|\int_{\R_{+}}Q_{1}dx\right|\leq C
\int_{\R}(\varphi^{2}+\zeta^{2})\left((\partial_{x}\theta^{cd})^{2}+|\partial^{2}_{x}\theta^{cd}|\right)dx
\leq
C\delta\|\varphi_{0}\|_{H^1}^{2}e^{-\frac{p_{-}}{\mu}t}+C\delta\|\partial_{x}[\varphi,\zeta]\|^{2},
\end{split}
\end{equation}
\begin{equation}\label{Q2}
\begin{split}
\left|\int_{\R_{+}}Q_{2}dx\right|\leq&
(C\epsilon_0+\eta)\|\pa_x\zeta\|^2+C_{\eta}
\int_{\R_{+}}(\varphi^{2}+\zeta^{2})(\partial_{x}\theta^{cd})^{2}dx
\\ \leq&
(C\epsilon_0+\eta)\|\pa_x\zeta\|^2+C_\eta\delta^2\|\varphi_{0}\|_{H^1}^{2}e^{-\frac{p_{-}}{\mu}t}+C_\eta\delta^2\|\partial_{x}[\varphi,\zeta]\|^{2}.
\end{split}
\end{equation}
For the terms involving $F$ and $G$, noticing that
$$
|\pa_tu^{cd}|=O(1)\de(1+t)^{-\frac{3}{2}} e^{-\frac{c_{1}x^{2}}{1+t}},\ |\pa^2_xu^{cd}|=|\pa_x\pa_tv^{cd}|=O(1)\de(1+t)^{-\frac{3}{2}} e^{-\frac{c_{1}x^{2}}{1+t}},\ \textrm{as}\ x\rightarrow+\infty,
$$
we get from Cauchy-Schwarz's inequality that
\begin{equation}\label{F.es}
\begin{split}
\left|\int_{\R_{+}}F\psi dx\right|\leq& \int_{\R_{+}}|\pa_tu^{cd}\psi| dx+C\int_{\R_{+}}\left|\pa^2_xu^{cd}\psi\right| dx
+C\int_{\R_{+}}\left|\pa_xu^{cd}\pa_xv^{cd}\psi\right| dx+
C\int_{\R_{+}}\left|\pa_xu^{cd}\pa_x\varphi\psi\right| dx\\
\leq&
C\de(1+t)^{-1-\al}\|\psi\|^2+C\de(1+t)^{-3/2+\al}+C\de\|\pa_x\varphi\|^2,
\end{split}
\end{equation}
where $0<\al<1/2$,
and
\begin{equation}\label{G.es}
\begin{split}
\left|\int_{\R_{+}}\frac{\zeta}{\theta}G dx\right|\leq C\|\zeta\|_{\infty}\|\pa_x
u\|^{2} \leq
C\epsilon_{0}\|\pa_x\psi\|^2+C\eps_0(1+t)^{-\frac{3}{2}}.
\end{split}
\end{equation}
We next compute the term $\widetilde{H}(0,t)$ arising from the boundary. Since $\zeta(0,t)=\si(0,t)=0$,
$\left|\widetilde{H}(0,t)\right|$ can be reduced to
\begin{equation*}\label{H.r}
\begin{split}
\left| \frac{R\ta_-\varphi(0,t)}{v(0,t)v_-}\psi(0,t)+\mu\left(\frac{\psi\pa_x\psi}{v}\right)(0,t)\right|,
\end{split}
\end{equation*}
which is further dominated by
\begin{equation}\label{H.es}
\begin{split}
 C&|\varphi(0,t)\psi(0,t)|+C|(\partial_{t}\varphi)(0,t)\psi(0,t)|
\\ \leq& C|\varphi_{0}(0)|e^{-\frac{p_{-}}{\mu}t}|\psi(0,t)|
\leq C\|\varphi_{0}(x)\|_{H^1}\|\psi\|^{1/2}\|\pa_x\psi\|^{1/2}e^{-\frac{p_{-}}{\mu}t}
%\\ \leq& C\|\varphi_{0}(x)\|^{4/3}_{H^1}e^{-\frac{p_{-}}{\mu}t}+C\|\psi\|^{2}\|\pa_x\psi\|^2
\\ \leq& C\|\varphi_{0}(x)\|^{4/3}_{H^1}e^{-\frac{p_{-}}{\mu}t}+C\eps^2_0\|\pa_x\psi\|^2,
\end{split}
\end{equation}
according to Lemma \ref{pbd}, Sobolev's inequality \eqref{sob.} and Young's inequaity.

In order to estimate $\CI_l$ $(1\leq l\leq 31)$, we first calculate
\begin{multline}\label{tIR}
I_0\sim \sigma^3,\ \ \pa_t I_0=-v^{cd}\pa_t\phi\int_{\phi^{cd}}^{\phi}(\phi-\varrho)\rho_{e}'''(\varrho)
d\varrho
+\frac{1}{2}\sigma^2\pa_t\phi^{cd}v^{cd}\rho_{e}'''(\phi^{cd})-\pa_tv^{cd}
\int_{\phi^{cd}}^{\phi}\frac{(\phi-\varrho)^{2}}{2}\rho_{e}'''(\varrho)
d\varrho
\\ \sim \pa_t\phi \sigma^2+\pa_tv^{cd} \sigma^2+\pa_tv^{cd}\sigma^3= \pa_t\sigma \sigma^2
+2\pa_tv^{cd}\sigma^2+\pa_tv^{cd}\sigma^3,
\end{multline}
and similarly,
\begin{equation}\label{xIR}
\pa_x I_0\sim\pa_x\sigma \sigma^2
+2\pa_xv^{cd}\sigma^2+\pa_xv^{cd}\sigma^3.
\end{equation}
In addition, from \eqref{pp} and \eqref{p.BC}, it follows
\begin{equation}\label{ptsi}
\|\pa_t\sigma\|^2+\|\pa_t\pa_x\sigma\|^2 \leq
C\|\pa_x\psi\|^2+C\eps_0\left\|\left[\pa_x\varphi,\pa^2_x\psi,
\pa_x\sigma,\pa^2_x\sigma\right]\right\|^2 +C\delta(1+t)^{-\frac{3}{2}}.
\end{equation}
For the sake of completeness, the proof of \eqref{ptsi} is given in the appendix.
%owing to Cauchy-Schwarz's inequality, \eqref{pv},
%\eqref{taylor},  \eqref{p.ine2} and Lemma \ref{es.tap}.

With \eqref{tIR}, \eqref{xIR} and \eqref{ptsi} in hand, we now
employ \eqref{p.ine2}, Cauchy-Schwarz's inequality with $0<\eta<1,$
Sobolev's inequality and Lemma \ref{es.tap} repeatedly to present
the following estimates:
\begin{equation*}\label{R1}
|\CI_1|\leq
C\int_{\R_{+}}\left|\pa_x[v^{cd},v]\pa_x\sigma\pa_t\sigma\right|dx \leq
C \eps_0\left\|\left[\pa_x\sigma,\pa^{2}_x\sigma,\pa_t\sigma\right]\right\|^2,
\end{equation*}
\begin{equation*}\label{R2}
\begin{split}
|\CI_2|\leq C&\int_{\R_{+}}\left|\pa^2_x\sigma\sigma\pa_x[\psi,
u^{cd}]\right|dx
+C\int_{\R_{+}}\left|\pa_x\sigma\sigma\pa_x[\varphi,v^{cd}]\pa_x[\psi, u^{cd}]\right|dx
\leq C
\eps_0\left\|\pa_x\left[\varphi,\psi,\sigma,\pa_x\sigma\right]\right\|^2,
\end{split}
\end{equation*}
\begin{equation*}\label{R3}
\begin{split}
|\CI_3|&+|\CI_5|+|\CI_{11}|+|\CI_{18}|\\
\leq&C\int_{\R_{+}}\left|\partial_{x}u^{cd}\sigma^{3}\right|dx
+C\int_{\R_{+}}\left|\sigma^{3}[\partial_{t}\sigma,\partial_{x}\psi]\right|
dx
+C\int_{\R_{+}}\left|\pa_x\psi\sigma^4\right|dx
+C\int_{\R}\left|\pa_xu^{cd}\sigma^4\right|dx\\
\leq& C\eps_0\left\|\left[\pa_t\sigma,\pa_x\sigma,\pa_x\psi\right]\right\|^2,
\end{split}
\end{equation*}
\begin{equation*}\label{R4}
\begin{split}
|\CI_4|\leq C\int_{\R_{+}}|\pa_t[v^{cd},v]||\pa_x\sigma|^2dx \leq
C\eps_0\left\|\pa_x[\sigma,\pa_x\sigma]\right\|^2,
\end{split}
\end{equation*}
%\begin{equation*}\label{R5}
%\begin{split}
%R_5\leq&C\int_{\R_{+}}\left|\pa_x\psi\sigma^4\right|dx
%+C\int_{\R}\left|\pa_xu^{cd}\sigma^4\right|dx
% \leq C \eps_0\|\pa_x[\psi,\sigma]\|^2,
%\end{split}
%\end{equation*}
\begin{equation*}\label{R7}
\begin{split}
|\CI_6|\leq&\int_{\R_{+}}\left|\pa_t\pa_xv^{cd}\left(\frac{\pa_x\phi^{cd}}{v}\right)\sigma
v^{-1}\right|dx
+\int_{\R_{+}}\left|\pa_tv^{cd}\left(\frac{\pa_x\phi^{cd}}{v}\right)\pa_x\sigma
v^{-1}\right|dx
\\&+\int_{\R_{+}}\left|\pa_tv^{cd}\left(\frac{\pa_x\phi^{cd}}{v}\right)\sigma \pa_x (v^{-1})\right|dx
+\int_{\R_{+}}\left|\pa_xv^{cd}\pa_t\left(\frac{\pa_x\phi^{cd}}{v}\right)\sigma
v^{-1}\right|dx
\\&+\int_{\R_{+}}\left|v^{cd}\pa_t\left(\frac{\pa_x\phi^{cd}}{v}\right)\pa_x\sigma v^{-1}\right|dx
+\int_{\R_{+}}\left|v^{cd}\pa_t\left(\frac{\pa_x\phi^{cd}}{v}\right)\sigma \pa_x (v^{-1})\right|dx\\
\leq &C
\delta(1+t)^{-\frac{3}{2}}+C\eps_0\|\pa_x[\varphi,\psi,\pa_x\sigma]\|^2,
\end{split}
\end{equation*}
\begin{equation*}\label{R8}
\begin{split}
|\CI_7|\leq C\|\psi\|_{L^\infty} (\|\pa_x\sigma\|^2+\|\pa_x^2\sigma\|^2)
\leq C \eps_0\|\pa_x[\sigma,\pa_x\sigma]\|^2,
\end{split}
\end{equation*}
\begin{equation*}\label{R9}
\begin{split}
|\CI_8|\leq &C\int_{\R_{+}}|\pa^2_x\sigma\sigma\pa_x\psi|dx+
C\int_{\R_{+}}|\pa^2_x\sigma\sigma\psi\pa_x\varphi|dx+C\int_{\R_{+}}|\pa^2_x\sigma\sigma\psi\pa_xv^{cd}|dx\\
\leq &C \eps_0\|\pa_x[\psi, \varphi,\pa_x\sigma]\|^2+C\epsilon_{0}
\int_{\R_{+}}(\partial_{x}\theta^{cd})^{2}\sigma^{2}dx
\leq C \eps_0\|\pa_x[\psi, \varphi,\si,\pa_x\sigma]\|^2,
\end{split}
\end{equation*}
\begin{equation*}\label{R10}
\begin{split}
|\CI_{9}|\leq& C\int_{\R_{+}}\left|\pa_x\sigma
\pa_xv\pa_x\sigma\psi\right|dx
+C\int_{\R_{+}}\left|\pa_x\sigma\pa_xv\sigma\pa_x\psi\right|dx+
C\int_{\R_{+}}\left|\pa_x\sigma\pa_xv\sigma\psi\pa_xv\right|dx
\\
\leq& C\|\pa_x\si\|_{H^1}\|\psi\|_{H^1}\|\pa_x\sigma\|\|\pa_xv\|
+C\|\pa_x\si\|_{H^1}\|\si\|_{H^1}\|\pa_x\psi\|\|\pa_xv\|
+C\|\pa_x\si\|_{H^1}\|\si\|_{H^1}\|\psi\|_{H^1}\|\pa_x\varphi\|^2
\\&+C\|\pa_x\si\|\|\pa_x\varphi\|\|\si\|_{H^1}\|\psi\|_{H^1}\|\pa_xv^{cd}\|_{H^1}
+C\|\psi\|_{H^1}\left(\|\pa_x\si\|^2+\|\si(\pa_xv^{cd})^2\|^2\right)
\\ \leq&
C\eps_0\|\pa_x[\sigma,\pa_x\sigma, \varphi, \psi]\|^2,
%+C\delta(1+t)^{-2},
\end{split}
\end{equation*}
\begin{equation*}\label{R101}
\begin{split}
|\CI_{10}|+|\CI_{29}|+|\CI_{30}|\leq C&\int_{\R_{+}}\left|
\pa_xu^{cd}\sigma^{2}\right|dx \leq
C\int_{\R_{+}}\left((\partial_{x}\theta^{cd})^{2}+|\partial^{2}_{x}\theta^{cd}|\right)\sigma^{2}dx\leq C\de\|\pa_x\si\|^2,
\end{split}
\end{equation*}
\begin{equation*}\label{R101}
\begin{split}
|\CI_{12}|&+|\CI_{13}|+|\CI_{14}|+|\CI_{19}|+|\CI_{20}|+|\CI_{25}|+|\CI_{26}|\\
\leq& C\int_{\R_{+}}\left|
\pa_xv^{cd}\sigma^{3}\psi\right|dx+C\int_{\R_{+}}\left|
\pa_x[\varphi,\sigma]\sigma^{3}\psi\right|dx\\
\leq& C\eps_0\|\pa_x[\sigma,\varphi]\|^2+C\eps_0
\int_{\R_{+}}(\partial_{x}\theta^{cd})^{2}\sigma^{2}dx
\leq C\eps_0\|\pa_x[\sigma,\varphi]\|^2,
\end{split}
\end{equation*}
\begin{equation*}\label{R16}
\begin{split}
|\CI_{15}|
\leq&\eta\int_{\R_{+}}\left|\pa^2_x\sigma\right|^2dx+C_\eta\int_{\R_{+}}\psi^{2}(\pa_x\theta^{cd})^{2}dx
+C\int_{\R_{+}}\left|\pa_x\sigma\pa_x\varphi\right|^2dx+C\int_{\R_{+}}\left|\pa_x\sigma\pa_xv^{cd}\right|^2dx
\\ \leq& (C\eps_0+\eta)\|\pa_x[\sigma,\pa_x\si]\|^2+C_\eta\int_{\R_{+}}\psi^{2}(\pa_x\theta^{cd})^{2}dx,
\end{split}
\end{equation*}
\begin{equation*}\label{R15}
\begin{split}
|\CI_{16}|\leq&C\int_{\R_{+}}\left|\pa^2_xv^{cd}\pa_x\sigma\psi\right|dx
+C\int_{\R_{+}}\left|\pa^2_xv^{cd}\sigma\pa_x\psi\right|dx\\
&+C\int_{\R_{+}}\left|\pa^2_xv^{cd}[\pa_x\varphi,\pa_xv^{cd}]\sigma\psi\right|dx
+C\int_{\R_{+}}\left|\pa_xv^{cd}[\pa_x\varphi,\pa_xv^{cd}]\pa_x\sigma\psi\right|dx
\\&+C\int_{\R_{+}}\left|\pa_xv^{cd}[\pa_x\varphi,\pa_xv^{cd}]\sigma\pa_x\psi\right|dx+
C\int_{\R_{+}}\left|\pa_xv^{cd}[\pa_x\varphi,\pa_xv^{cd}]^2\sigma\psi\right|dx\\
\leq& C\eps_0\|\pa_x[\sigma,\varphi,
\psi]\|^2+C\delta(1+t)^{-\frac{3}{2}},
\end{split}
\end{equation*}
\begin{equation*}\label{R18}
\begin{split}
|\CI_{17}|+|\CI_{27}|\leq&C\int_{\R_{+}}\left|(\pa_xv^{cd})^{2}\pa_xv\psi\right|dx
+C\int_{\R_{+}}\left|\pa^2_xv^{cd}\pa_xv^{cd}\psi\right|dx
\\ \leq&
C\eps_0\|\pa_x\varphi\|^2+C\int_{\R_{+}}\psi^{2}(\pa_x\theta^{cd})^{2}dx+C\de^2(1+t)^{-\frac{3}{2}},
\end{split}
\end{equation*}
\begin{equation*}\label{R18}
\begin{split}
|\CI_{21}|+|\CI_{22}|+|\CI_{23}|+|\CI_{24}|\leq&C\int_{\R_{+}}\left|\pa_x[\sigma,\varphi]\sigma^2\psi\right|dx+C\int_{\R_{+}}\left|\sigma^2\psi\varphi\pa_xv^{cd}\right|dx\\
\leq& C\eps_0\|\pa_x[\sigma,\varphi]\|^2+C\eps_0
\int_{\R_{+}}(\partial_{x}\theta^{cd})^{2}\varphi^{2}dx
\\ \leq& C\eps_0\|\pa_x[\sigma,\varphi]\|^2+
C\delta^2\|\varphi_{0}\|_{H^1}^{2}e^{-\frac{p_{-}}{\mu}t},
\end{split}
\end{equation*}
\begin{equation*}\label{R16}
\begin{split}
|\CI_{28}|\leq&C\int_{\R_{+}}\left|\pa^{2}_x\sigma\pa_xv^{cd}\sigma\psi\right|dx
+C\int_{\R_{+}}\left|\pa_x\sigma\pa_xv\pa_xv^{cd}\psi\sigma\right|dx\\
\leq& C\eps_0\|\pa_x[\sigma,\pa_x\sigma,\varphi]\|^2+C\eps_0
\int_{\R_{+}}(\partial_{x}\theta^{cd})^{2}\sigma^{2}dx\leq C\eps_0\|\pa_x[\sigma,\pa_x\sigma,\varphi]\|^2.
\end{split}
\end{equation*}
For the last term $\CI_{31}$, applying \eqref{taylor} again, one can see that
\begin{equation*}\label{I22}
\begin{split}
\CI_{31}
=&-\frac{3}{2}\int_{\R_{+}}v^{cd}\rho_{e}'(\phi^{cd})\pa_t\left(\frac{\sigma^2}{v}\right)
\left(-v^{cd}\pa_x\left(\frac{\pa_x\sigma}{v}\right)
+v\left(1-v^{cd}\rho_{e}(\sigma+\phi^{cd})\right)
-v^{cd}\pa_x\left(\frac{\pa_x\phi^{cd}}{v}\right)\right)dx,
\end{split}
\end{equation*}
which implies
\begin{equation*}\label{I22}
\begin{split}
\left|\CI_{31}
-\frac{d}{dt}\int_{\R_{+}}\left(v^{cd}\rho_{e}'(\phi^{cd})\right)^{2}\sigma^3dx\right|
\leq C\eps_0\left\|\left[\pa_x\psi,\pa_x\varphi,\pa_t\sigma,
\pa_x\sigma,\pa_x^2\sigma\right]\right\|^2 +C\delta(1+t)^{-\frac{3}{2}}.
\end{split}
\end{equation*}
Let us now define $\si_0(x)=\si(x,0)=\phi(x,0)-\phi^{cd}(x,0)$.
From  the Poisson equation \eqref{pp}, it follows that for any $t\geq0$
\begin{eqnarray*}
\|\si(t)\|_{H^1}^2&\leq &C\|\varphi(t)\|^2 +C\left\|\pa_x^2 \phi^{cd}(t)\right\|^2+C\left\|\left(\pa_x \phi^{cd}\right)^2(t)\right\|^2
\leq C\|\varphi(t)\|^2 +C\de^2,
\end{eqnarray*}
and hence in particular,
\begin{equation}\label{si.ides}
\|\sigma_0\|_{H^1}^2\leq C\|\varphi_0\|^2 +C\de^2.
\end{equation}

We now conclude from \eqref{zs2}, \eqref{Q1}, \eqref{Q2}, \eqref{F.es}, \eqref{G.es}, \eqref{H.es}, \eqref{ptsi}, \eqref{si.ides} and the above
estimates on $\CI_{l}$ $(1\leq l\leq 31)$ that
\begin{equation}\label{zeng.p3}
\begin{split}
\left\|[\psi,\varphi,\zeta]\right\|^2&+\left\|\sigma\right\|^2_{H^1}+\eps_0\left\|\partial_{x}\varphi\right\|^2
+\int_{0}^{T}\|\pa_x\left[\psi,\zeta\right]\|^2dt\\
\leq& C\left\|[\psi_{0},\zeta_{0}]\right\|^2+C\left\|\varphi_{0}\right\|^{4/3}_{H^1}
+(C\eps_0+\eta)\int_{0}^{T}\|\pa_x[\varphi,\pa_x\psi,\sigma,\pa_x\sigma]\|^2dt
\\&+C_\eta\int_{0}^{T}\int_{\R_{+}}\psi^{2}(\pa_x\theta^{cd})^{2}dxdt
+C\delta,
\end{split}
\end{equation}
for suitably small $\eps_0>0$, $\de>0$  and $\eta>0$.

{\bf Step 2.} {\it Dissipation of
$\pa_x[\varphi,\sigma,\pa_x\sigma]$.}

We first differentiate
 \eqref{pv} with respect to $x$,
to obtain
\begin{equation}\label{dpp}
\pa_t\pa_x\varphi-\pa^2_x\psi=0.
\end{equation}
Then multiplying \eqref{pp}, \eqref{pu} and \eqref{dpp} by
$\pa^{2}_x\sigma$, $-v\pa_x\varphi$ and $\mu\pa_x\varphi$, respectively, and
integrating the resulting equalities with respect to $x$ over $\R_{+}$, one has
\begin{equation}\label{pp.ip1}
\begin{split}
\int_{\R_{+}}&\frac{v^{cd}}{v}(\pa^2_x\sigma)^{2}dx
+\int_{\R_{+}}vv^{cd}|\rho_{e}'(\phi^{cd})|(\pa_x\sigma)^{2}dx\\
=&\int_{\R_{+}}\frac{v^{cd}}{v^2}\pa_xv\pa_x\sigma\pa^2_x\sigma
dx-\int_{\R_{+}}\pa_xv\left[1-v^{cd}\rho_{e}(\sigma+\phi^{cd})\right]\pa_x\sigma
dx
+\int_{\R_{+}}v\pa_x[v^{cd}\rho_{e}'(\phi^{cd})]\sigma\pa_x\sigma dx\\
&+\int_{\R_{+}}v\pa_x\left[\frac{v^{cd}\rho_{e}''(\phi^{cd})}{2}\sigma^{2}\right]\pa_x\sigma
dx+\int_{\R_{+}}\pa_x\varphi\pa_x\sigma
dx+\varphi(0,t)\pa_x\si(0,t)-\int_{\R_{+}}v\pa_xI_0\pa_x\sigma
dx
\\&
-\int_{\R_{+}}v^{cd}\pa_x\left(\frac{\pa_x\phi^{cd}}{v}\right)\pa^{2}_x\sigma
dx,
\end{split}
\end{equation}
\begin{equation}\label{pu.ip1}
\begin{split}
-&\int_{\R_{+}}\pa_t \psi v\pa_x\varphi dx-\int_{\R_{+}}(\pa_x
p-\pa_xp^{cd})v\pa_x\varphi dx+\int_{\R_{+}}\pa_x\varphi\pa_x\sigma dx
+\int_{\R_{+}}\left(\frac{\pa_x\phi^{cd}}{v}-\frac{\pa_x\phi^{cd}}{v^{cd}}\right)v\pa_x\varphi dx\\
&=-\int_{\R_{+}}\mu\pa^2_x \psi\pa_x\varphi dx
-\int_{\R_{+}}\mu\pa^2_x u^{cd}\pa_x\varphi
dx-\mu\int_{\R_{+}}\pa_x(v^{-1})\pa_x uv\pa_x\varphi
dx+\int_{\R_{+}}\pa_tu^{cd}v\pa_x\varphi dx,
\end{split}
\end{equation}
and
\begin{equation}\label{pv.ip1}
\int_{\R_{+}}\mu\left(\pa_t\pa_x\varphi-\pa^2_x\psi\right)\pa_x\varphi
dx=0.
\end{equation}
%where in \eqref{pp.ip1} we have used the boundary condition $\si(0,t)=\si(+\infty,t)=0$.
The summation of \eqref{pp.ip1}, \eqref{pu.ip1} and
\eqref{pv.ip1} further implies
\begin{equation}\label{sum.ip1}
\begin{split}
-\frac{d}{dt}&\int_{\R_{+}}\psi v\pa_x\varphi dx
+\frac{\mu}{2}\frac{d}{dt}\int_{\R_{+}}(\pa_x\varphi)^{2}dx+\int_{\R_{+}}p^{cd}(\pa_x\varphi)^{2}dx
\\&+\int_{\R_{+}}\frac{v^{cd}}{v}(\pa^2_x\sigma)^{2}dx
+\int_{\R_{+}}vv^{cd}|\rho_{e}'(\phi^{cd})|(\pa_x\sigma)^{2}dx\\
=&\underbrace{-\int_{\R_{+}}\psi\pa_tv\pa_x\varphi dx}_{J_1}\underbrace{- \int_{\R_{+}}\psi
v\pa_t\pa_x\varphi dx}_{J_2}\underbrace{+\int_{\R_{+}}R\pa_x
\left[\frac{\zeta}{v}\right]v\pa_x\varphi dx}_{J_3}
\underbrace{-\int_{\R_{+}}R\varphi\pa_x
\left[\frac{\ta^{cd}}{vv^{cd}}\right]v\pa_x\varphi dx}_{J_4}\\
&\underbrace{-\int_{\R_{+}}\left(\frac{\pa_x\phi^{cd}}{v}-\frac{\pa_x\phi^{cd}}{v^{cd}}\right)v\pa_x\varphi
dx}_{J_5} \underbrace{-\int_{\R_{+}}\mu\pa^2_x u^{cd}\pa_x\varphi
dx}_{J_6}\underbrace{-\mu\int_{\R_{+}}\pa_x(v^{-1})\pa_x uv\pa_x\varphi dx}_{J_7}\\
&\underbrace{+\int_{\R_{+}}\frac{v^{cd}}{v^2}\pa_xv\pa_x\sigma\pa^2_x\sigma
dx}_{J_8}\underbrace{-\int_{\R_{+}}\pa_xv\left[1-v^{cd}\rho_{e}(\sigma+\phi^{cd})\right]\pa_x\sigma
dx}_{J_9}\underbrace{+\varphi(0,t)\pa_x\si(0,t)}_{J_{10}}\\& \underbrace{-\int_{\R_{+}}v\pa_xI_0\pa_x\sigma dx}_{J_{11}}
\underbrace{+\int_{\R_{+}}v\pa_x\left[v^{cd}\rho_{e}'(\phi^{cd})\right]\sigma\pa_x\sigma
dx}_{J_{12}}\underbrace{+\int_{\R_{+}}v\pa_x\left[\frac{v^{cd}\rho_{e}''(\phi^{cd})}{2}\sigma^{2}\right]\pa_x\sigma
dx}_{J_{13}}\\&\underbrace{+\int_{\R_{+}}\pa_tu^{cd}v\pa_x\varphi
dx}_{J_{14}}\underbrace{-\int_{\R_{+}}v^{cd}\pa_x\left(\frac{\pa_x\phi^{cd}}{v}\right)\pa^2_x\sigma
dx}_{J_{15}}.
\end{split}
\end{equation}
We now turn to compute $J_l$ $(1\leq l\leq 15)$ term by term. For brevity, we directly
give the following computations:
\begin{equation*}
\begin{split}
|J_{1}|\leq&C\int_{\R_{+}}\left|\psi\pa_x\psi\pa_x\varphi\right|
dx+C\int_{\R_{+}}\left|\psi\pa_xu^{cd}\pa_x\varphi\right|dx\leq
C\eps_0\|\pa_x[\psi,\varphi]\|^2+C\delta(1+t)^{-3/2},
\end{split}
\end{equation*}
\begin{equation*}
\begin{split}
|J_{2}|\leq&\left|\psi(0,t)v(0,t)(\pa_t\varphi)(0,t)\right|+\left|\int_{\R_{+}}v(\pa_x\psi)^{2}dx\right|+
\left|\int_{\R_{+}}\psi\pa_xv\pa_x\psi dx\right|\\
\leq&C\|\varphi_{0}\|_{H^1}^{\frac{4}{3}}e^{-\frac{p_{-}}{\mu}t}+
C\eps_0\|\pa_x[\psi,\varphi]\|^2+C\|\pa_x\psi\|^2
+C\int_{\R_{+}}\psi^{2}(\pa_x\theta^{cd})^{2}dx,
\end{split}
\end{equation*}
%where we have used the same method as \eqref{F1fg*bn} to deal with
%$J_{2}$ and $J_{3}$.
\begin{equation*}
\begin{split}
|J_{3}|+|J_{4}|+|J_{5}| &\leq
(\eta+C\eps_0)\|\pa_x\varphi\|^2+C_{\eta}\|\pa_x\zeta\|^2
+C_{\eta}\int_{\R_{+}}(\varphi^{2}+\zeta^{2})(\pa_x\theta^{cd})^{2}dx\\
&\leq(\eta+C\eps_0)\|\partial_{x}[\varphi,\zeta]\|^2+C_{\eta}\|\pa_x\zeta\|^2
+C\delta\|\varphi_{0}\|_{H^1}^{2}e^{-\frac{p_{-}}{\mu}t},
\end{split}
\end{equation*}
\begin{equation*}
\begin{split}
|J_{6}| \leq C\delta\|\pa_x\varphi\|^2+C\delta(1+t)^{-5/2},
\end{split}
\end{equation*}
\begin{equation*}
\begin{split}
|J_{7}|\leq& C\int_{\R_{+}}(\pa_x\varphi)^{2}|\pa_x\psi| dx
+C\int_{\R_{+}}(\pa_x\varphi)^{2}|\pa_xu^{cd}|
dx+C\int_{\R_{+}}|\pa_xv^{cd}\pa_xu^{cd}\pa_x\varphi| dx+
\int_{\R_{+}}|\pa_xv^{cd}\pa_x\psi\pa_x\varphi|dx\\
\leq&C\eps_0\|\pa_x[\psi,\pa_x\psi,\varphi]\|^2
+C\delta(1+t)^{-2},
\end{split}
\end{equation*}
\begin{equation*}
|J_{8}|+|J_{9}| \leq C\eps_0\|\pa_x[\varphi,\pa_x\sigma,\sigma]\|^2,
\end{equation*}
\begin{equation*}
\begin{split}
|J_{10}|\leq
C\eps_0\|\pa_x\sigma\|_{H^1}^2+C\|\varphi_{0}\|_{H^1}^{\frac{4}{3}}e^{-\frac{4p_{-}}{3\mu}t},
\end{split}
\end{equation*}
\begin{equation*}
\begin{split}
|J_{11}|+|J_{12}|+|J_{13}|+|J_{14}|+|J_{15}|\leq
C\eps_0\|\pa_x[\varphi,\sigma,\pa_x\si]\|^2+C\delta(1+t)^{-3/2}.
\end{split}
\end{equation*}
%\begin{equation*}
%|J_{15}| \leq
%C\delta\|\pa_x[\varphi,\pa_x\sigma]\|^2+C\delta(1+t)^{-\frac{3}{2}}.
%\end{equation*}
Substituting the above estimations for $J_l$ $(1\leq l\leq 15)$ into
\eqref{sum.ip1}, letting $\eta>0$ be suitably small and combing \eqref{zeng.p3}, we obtain
\begin{equation}\label{sum.ip2}
\begin{split}
\left\|[\psi,\varphi,\zeta]\right\|^2&+\left\|\sigma\right\|^2_{H^1}+
\|\pa_x\varphi\|^2
+\int_{0}^{T}\|\pa_x\sigma\|^2_{H^1}dt
+\int_{0}^{T}\|\pa_x\left[\varphi,\psi,\zeta\right]\|^2dt
\\
\leq&C\eps_0\int_{0}^{T}\|\pa^{2}_x\psi\|^2dt+C\delta+
C\left\|[\psi_{0},\zeta_{0}]\right\|^2+C\left\|\varphi_{0}\right\|^{4/3}_{H^1}
+C\int_{0}^{T}\int_{\R_{+}}\psi^{2}(\pa_x\theta^{cd})^{2}dxdt.
\end{split}
\end{equation}
%where  $\eps_0>0$, $\eta>0$ are chosen to be suitable constants.

{\bf Step 3.} {\it Higher order energy estimates.}

 Multiplying
\eqref{pu} by $-\pa^{2}_x\psi$, and integrating the resultant
equality with respect to $x$ over $\R_{+}$, one has
\begin{equation}\label{sum.ip3}
\begin{split}
\frac{1}{2}\frac{d}{dt}&\int_{\R_{+}}(\pa_x\psi)^{2}dx+\mu\int_{\R_{+}}\frac{(\pa^{2}_x\psi)^{2}}{v}dx
\\
=&\underbrace{-\int_{\R_{+}}\left(\frac{\pa_x\phi}{v}-\frac{\pa_x\phi^{cd}}{v^{cd}}\right)\pa^{2}_x\psi
dx}_{J_{16}}\underbrace{+\int_{\R_{+}}\pa_x(p-p^{cd})\pa^{2}_x\psi
dx}_{J_{17}}\underbrace{+\mu\int_{\R_{+}}\frac{\pa_x\psi\pa_x\varphi}{v^{2}}\pa^{2}_x\psi
dx}_{J_{18}}\\&\underbrace{+\mu\int_{\R_{+}}\frac{\pa_x\psi\pa_xv^{cd}}{v^{2}}\pa^{2}_x\psi
dx}_{J_{19}}\underbrace{-\int_{\R_{+}}F\pa^{2}_x\psi
dx}_{J_{20}}\underbrace{-(\pa_t\psi\pa_t\varphi)(0,t)}_{J_{21}}.
%=\sum\limits_{l=31}^{36}I_{l},
\end{split}
\end{equation}
%where we have used boundary condition $\psi(0,t)=0.$
To obtain the estimates for $J_l$ $(16\leq l\leq 21)$, we
use Cauchy-Schwarz's inequality with $0<\eta<1$, Sobolev's inequlity \eqref{sob.} and
\eqref{p.ine2} repeatedly to perform the calculations as follows:
\begin{equation*}\label{3.2}
\begin{split}
|J_{16}|\leq&
C\int_{\R_{+}}\left|\pa_xv^{cd}\varphi\pa^{2}_x\psi\right|dx+C\int_{\R_{+}}\left|\pa_x\sigma\pa^{2}_x\psi\right|dx\\
\leq&
(C\delta+\eta)\|\pa_x[\varphi,\pa_x\psi]\|^{2}+C_{\eta}\|\pa_x\sigma\|^{2}+C\delta\|\varphi_{0}\|_{H^1}^{2}e^{-\frac{p_{-}}{\mu}t},
\end{split}
\end{equation*}
\begin{equation*}\label{3.3}
\begin{split}
|J_{17}|\leq&
C\int_{\R_{+}}\left|\pa_x[\zeta,\varphi]\pa^{2}_x\psi\right|dx
+C\int_{\R_{+}}\left|[\zeta,\varphi]\pa_x[\varphi,v^{cd}]\pa^{2}_x\psi\right|dx\\
\leq&
(C\eps_0+\eta)\|\pa_x[\varphi,\pa_x\psi]\|^{2}+C_{\eta}\|\pa_x[\zeta,\varphi]\|^{2}+C\delta\|\varphi_{0}\|_{H^1}^{2}e^{-\frac{p_{-}}{\mu}t},
\end{split}
\end{equation*}
\begin{equation*}\label{3.4}
\begin{split}
|J_{18}|+|J_{19}|\leq C\eps_0\|\pa_x[\psi,\pa_x\psi]\|^{2},
\end{split}
\end{equation*}

\begin{equation*}\label{3.6}
\begin{split}
|J_{20}|\leq&
C\int_{\R_{+}}\left|\pa^{2}_xu^{cd}\pa^{2}_x\psi\right|dx+
C\int_{\R_{+}}\left|\pa_xu^{cd}\pa_xv\pa^{2}_x\psi\right|dx+C\int_{\R_{+}}\left|\pa_tu^{cd}\pa^{2}_x\psi\right|dx\\
\leq&
C\delta\|\pa_x[\varphi,\pa_x\psi]\|^{2}+C\delta(1+t)^{-\frac{5}{2}}.
\end{split}
\end{equation*}
For the last term $J_{21}$, in light of Lemma \ref{pbd}, we have
\begin{equation}\label{J21}
\begin{split}
J_{21}=-(\pa_t\psi\pa_t\varphi)(0,t)=-\pa_t[(\psi\pa_t\varphi)(0,t)]+\varphi_{0}(0)\frac{(p_{-})^{2}}{\mu^{2}}\psi(0,t)
e^{-\frac{p_{-}}{\mu}t},
%\\
%&\leq-\pa_t[(\psi\pa_t\varphi)(0,t)]+
%C\eps_0\|\pa_x\psi\|^2+C\|\varphi_{0}\|_{H^1}^{\frac{4}{3}}e^{-\frac{4p_{-}}{3\mu}t}
\end{split}
\end{equation}
furthermore, it follows that
\begin{equation}\label{J211}
\begin{split}
|(\psi\pa_t\varphi)(0,T)|\leq C\varphi_{0}(0)\psi(0,T)
e^{-\frac{p_{-}}{\mu}T}\leq C\eps_0\|\varphi_{0}\|_{H^1}e^{-\frac{p_{-}}{\mu}T},
\end{split}
\end{equation}
and
\begin{equation}\label{J212}
\begin{split}
|(\psi\pa_t\varphi)(0,0)|\leq C|\psi_{0}(0)\varphi_0(0)|\leq C(\|\psi_0\|^2_{H^1}+\|\varphi_{0}\|^2_{H^1}).
\end{split}
\end{equation}
By virtue of \eqref{J21}, \eqref{J211} and \eqref{J212} and carrying out the similar calculations as \eqref{H.es}, we thereby obtain
\begin{equation*}\label{3.6}
\begin{split}
\left|\int_{0}^{T}J_{21}dt\right| \leq C\eps_0\int_{0}^{T}\|\pa_x\psi\|^2dt+C\|\varphi_{0}\|_{H^1}^{\frac{4}{3}}+C\|\psi_0\|^2_{H^1}+C\eps_0\|\varphi_{0}\|_{H^1}.
\end{split}
\end{equation*}

Plug the above estimations for $J_l$ $(16\leq l\leq 21)$ into
\eqref{sum.ip3}, and recall \eqref{sum.ip2} and \eqref{zeng.p3}, then
choose $\eps_0>0$, $\delta>0$ and $\eta>0$ suitably small, to derive
\begin{equation}\label{sum.ip4}
\begin{split}
\left\|[\psi,\varphi,\zeta]\right\|^2&+\left\|\sigma\right\|^2_{H^1}
+\|\pa_x\varphi\|^2+\|\pa_x\psi\|^2\\
&+\int_{0}^{T}\|\pa_x\sigma\|^2_{H^1}dt
+\int_{0}^{T}\|\pa_x\left[\varphi,\psi,\zeta\right]\|^2dt+\int_{0}^{T}\|\pa^{2}_x\psi\|^{2}dt\\
\leq& C\delta+
C\left\|\zeta_{0}\right\|^2+C\left\|[\psi_{0},\varphi_{0}]\right\|^{4/3}_{H^1}
+C\int_{0}^{T}\int_{\R_{+}}\psi^{2}(\pa_x\theta^{cd})^{2}dxdt.
\end{split}
\end{equation}
Similarly, multiplying \eqref{pta} by $-\pa^{2}_x\zeta$, and
integrating the resulting equality over $\R_{+}$, we obtain
\begin{equation}\label{sum.ip5}
\begin{split}
\frac{R}{2(\gamma-1)}&\frac{d}{dt}\int_{\R_{+}}(\pa_x\zeta)^{2}dx+\kappa\int_{\R_{+}}\frac{(\pa^{2}_x\zeta)^{2}}{v}dx
\\
=&\int_{\R_{+}}(p\pa_xu-p^{cd}\pa_xu^{cd})\pa^{2}_x\zeta
dx+\kappa\int_{\R_{+}}\frac{\pa_x\zeta\pa_x\varphi}{v^{2}}\pa^{2}_x\zeta
dx+\kappa\int_{\R_{+}}\frac{\pa_x\zeta\pa_xv^{cd}}{v^{2}}\pa^{2}_x\zeta
dx\\&+\kappa\int_{\R_{+}}\partial_{x}\left(\frac{\varphi\partial_{x}\theta^{cd}}{vv^{cd}}\right)
\pa^{2}_x\zeta dx-\int_{\R_{+}}G\pa^{2}_x\zeta
dx,
\end{split}
\end{equation}
where we have used boundary condition $\zeta(0,t)=0.$
The right hand side of \eqref{sum.ip5} can be handled as $J_{l}$ $(16\leq l\leq 21)$, the details of which we omit, therefore one can get from \eqref{sum.ip5} and \eqref{sum.ip4} that
\begin{equation}\label{sum.ip6}
\begin{split}
\left\|[\psi,\varphi,\zeta]\right\|_{H^1}^2&+\left\|\sigma\right\|^2_{H^1}
+\int_{0}^{T}\|\pa_x\sigma\|^2_{H^1}dt
+\int_{0}^{T}\|\pa_x\left[\varphi,\psi,\zeta\right]\|^2dt+\int_{0}^{T}\|\pa^{2}_x[\psi,\zeta]\|^{2}dt\\
\leq& C\delta+
C\left\|[\varphi_{0},\psi_{0},\zeta_0]\right\|^{4/3}_{H^1}
+C\int_{0}^{T}\int_{\R_{+}}\psi^{2}(\pa_x\theta^{cd})^{2}dxdt.
\end{split}
\end{equation}
Finally, letting $\de>0$ small enough, combing \eqref{sum.ip6} and \eqref{c.eng} in Lemma \ref{key.es}, we obtain \eqref{eng.p1} as desired, this completes the proof of Proposition \ref{ape}.

\end{proof}

\section{Global existence and large time behavior}

We are now in a position to complete the proof of Theorem 1.1. %proof of Theorem \ref{main.Res.}.

\begin{proof}[Proof of Theorem \ref{main.res.}]
In view of the energy estimates obtained in Proposition \ref{ape}, one sees that
\begin{equation}\label{eng.p2}
\begin{split}
\sup_{0\leq t\leq
T}&\left\|[\varphi,\psi,\zeta,\sigma](t)\right\|_{H^1}^2
\leq
C\delta+C\left\|[\psi_{0},\zeta_{0},\varphi_{0}]\right\|^{4/3}_{H^1}.
\end{split}
\end{equation}
Notice that $\de>0$ is a parameter independent of $\eps_0$. By letting $\de>0$ be small enough, the global existence of
the solution of the Cauchy problem \eqref{pv}, \eqref{pu}, \eqref{pta}, \eqref{pp}, \eqref{p.BC} and \eqref{p.id}
 then follows from the standard continuation argument based on the local existence (cf. \cite{HMS-04}) and the {\it a
 priori}
estimate \eqref{eng.p1}. Moreover, \eqref{eng.p2} implies
\eqref{main.eng}. Our intention next is to prove the large time
behavior as \eqref{sol.lag}. For this, we first justify the
following limits:
\begin{equation}\label{latm1}
\lim\limits_{t\rightarrow+\infty}\left\|\pa_x[\varphi,\psi,\zeta](t)
\right\|_{L^2}^2= 0,
\end{equation}
and
\begin{equation}\label{latm2}
\lim\limits_{t\rightarrow+\infty}\left\|\pa_x\sigma(t)\right\|^2
= 0.
\end{equation}
To prove \eqref{latm1} and \eqref{latm2}, we get from \eqref{pv}, \eqref{pu}, \eqref{pta}, \eqref{eng.p1} and \eqref{tac.pt} that
\begin{equation}\label{latm3}
\begin{split}
\int_{0}^{+\infty}\left|\frac{d}{dt}\left\|\pa_x[\varphi,\psi,\zeta]\right\|^2\right|dt
=&2\int_{0}^{+\infty}\left|\left(\pa_t\pa_x[\varphi,\psi,\zeta],\pa_x[\varphi,\psi,\zeta]\right)\right|dt
\\ \leq& C+C\int_{0}^{+\infty}\left\|\pa_x\left[\varphi,\psi,\zeta,\si,\pa_x\left[\psi,\zeta,\si\right]\right]\right\|^2dt<+\infty.
\end{split}
\end{equation}
On that other hand, \eqref{ptsi}, \eqref{sum.ip2}  and \eqref{eng.p1} yield
\begin{equation}\label{latm4}
\int_{0}^{+\infty}\left|\frac{d}{dt}\left\|\pa_x\sigma\right\|^2\right|dt
=2\int_{0}^{+\infty}\left|\left(\pa_t\pa_x\sigma,\pa_x\sigma\right)\right|dt<+\infty.
\end{equation}
Consequently, \eqref{latm3}, \eqref{latm4} together with \eqref{eng.p1} gives \eqref{latm1} and \eqref{latm2}.
Then \eqref{sol.lag} follows from \eqref{latm1}, \eqref{latm2} and Sobolev's inequality \eqref{sob.}.
This ends the proof of Theorem \ref{main.res.}.

\end{proof}

\section{Appendix}
In this appendix, we will give some basic results used in the paper.
The first lemma is borrowed from \cite{HMS-04}.
\begin{lemma}\label{es.tap}
Let $\ta^{cd}$ satisfy \eqref{pro.eqn}, for $|\theta_{+}-\theta_{-}|=\delta$, it holds that
\begin{equation*}\label{1.29gsg8*x}
\begin{split}
\int_{\R_{+}}(\partial_{x}\theta^{cd})^{4}dx\leq
C\delta^{4}(1+t)^{-\frac{3}{2}},\ \
\int_{\R_{+}}(\partial^{2}_{x}\theta^{cd})^{2}dx\leq
C\delta^{2}(1+t)^{-\frac{3}{2}},
\end{split}
\end{equation*}
\begin{equation*}\label{1.29g8ko*x}
\begin{split}
\int_{\R_{+}}(\partial^{3}_{x}\theta^{cd})^{2}dx\leq
C\delta^{2}(1+t)^{-\frac{5}{2}},\ \
\int_{\R_{+}}x\left((\partial_{x}\theta^{cd})^{2}+|\partial^{2}_{x}\theta^{cd}|\right)dx
\leq C\delta.
\end{split}
\end{equation*}
\end{lemma}
Next is the key observation from the boundary condition \eqref{p.BC}.
\begin{lemma}\label{pbd}
It holds that
\begin{equation}\label{pvbd}
\varphi(0,t)=\varphi_{0}(0)e^{-\frac{p_{-}}{\mu}t}.
\end{equation}
\end{lemma}
\begin{proof}
%We first estimate the value of $\varphi(0,t)$ on the boundary $x=0$
%by the boundary condition \eqref{1.289gh*x} and \eqref{34tphy}.
%Setting $\varphi(t)=\varphi(0,t)$.
Since
$\partial_{x}u^{cd}(0,t)=0,$ from
\eqref{p.BC} it follows that
\begin{equation*}\label{resultqa12e}
\begin{split}
\frac{R\theta_{-}}{v_{-}+\varphi(0,t)}-\mu\frac{\partial_{t}\varphi(0,t)}{v_{-}+\varphi(0,t)}=p_{-},
\ \ t>0,
\end{split}
\end{equation*}
which implies
\begin{equation}\label{pvbd2}
\partial_{t}\varphi(0,t)=-\frac{p_{-}}{\mu}\varphi(0,t).
\end{equation}
\eqref{pvbd} follows from \eqref{pvbd2} and the compatibility condition $\varphi(0,0)=\varphi_{0}(0)$. This ends the proof of Lemma \ref{pbd}.

\end{proof}
We now give the following estimates concerning the delicate term
$\displaystyle{\int_{0}^{T}\int_{\R_{+}}}(\partial_{x}\theta^{cd})^{2}\psi^{2}dx
dt$.
\begin{lemma}\label{key.es}
Assume all the conditions listed in Proposition \ref{ape} hold, then for any $0\leq T\leq +\infty$,
there exists an energy functional $\CE(\varphi,\psi,\zeta)$ with
$$
|\CE(\varphi,\psi,\zeta)|\leq C\de^2\|[\varphi,\psi,\zeta]\|^2,
$$
such that the following energy estimate holds
\begin{equation}\label{c.eng}
\begin{split}
\CE(\varphi,\psi,\zeta)(T)+\int_{0}^{T}\int_{\R_{+}}(\partial_{x}\theta^{cd})^{2}\psi^{2}dx dt
\leq C\delta+C\delta\left\|\varphi_{0}\right\|^{2}_{H^1}
+C\delta\int_{0}^{T}\|\partial_{x}[\varphi,\psi,\zeta,\sigma,\partial_{x}\sigma\|^{2}dt.
\end{split}
\end{equation}
\end{lemma}
\begin{proof}
Define $$w=\int_{0}^{x}(\partial_{y}\theta^{cd})^{2}dy.$$
It is easy to check that
\begin{equation}\label{w.es}
\begin{split}
\|w(\cdot,t)\|_{\infty}\leq C\delta^{2} (1+t)^{-\frac{1}{2}},\ \
\|\partial_{t}w(\cdot,t)\|_{\infty}\leq
C\delta^{2}(1+t)^{-\frac{3}{2}}.
\end{split}
\end{equation}
From \eqref{pp} and \eqref{taylor}, it follows that
\begin{equation}\label{si.ep}
\begin{split}
\sigma=-\frac{\varphi}{vv^{cd}\rho_{e}'(\phi^{cd})}
\underbrace{-\frac{1}{vv^{cd}\rho_{e}'(\phi^{cd})}\left[v^{cd}\pa_x
\left(\frac{\pa_x\sigma}{v}\right)+\left(\frac{v^{cd}\rho_{e}^{''}(\phi^{cd})}{2}\sigma^{2}-I_0\right)v
+v^{cd}\pa_x\left(\frac{\pa_x\phi^{cd}}{v}\right)\right]}_{\CM}.
\end{split}
\end{equation}
On the other hand, \eqref{pu} can be rewritten as
\begin{equation}\label{pp.rw}
\begin{split}
\pa_t \psi+\pa_x\left(\frac{R\zeta-p^{cd}\varphi-\sigma}{v}\right)
=-\pa_x\left(\frac{1}{v}\right)\sigma-\frac{\varphi\pa_x\phi^{cd}}{vv^{cd}}+\mu\pa_x\left(\frac{\pa_x
\psi}{v}\right)+F.
\end{split}
\end{equation}
Substituting \eqref{si.ep} into \eqref{pp.rw}, one has
\begin{equation}\label{pp.rw2}
\begin{split}
\pa_t
\psi+\pa_x\left(\frac{R\zeta+\left(\frac{1}{vv^{cd}\rho_{e}'(\phi^{cd})}-p^{cd}\right)\varphi}{v}\right)
=\pa_x\left(\frac{\CM}{v}\right)-\pa_x\left(\frac{1}{v}\right)\sigma-\frac{\varphi\pa_x\phi^{cd}}{vv^{cd}}+\mu\pa_x\left(\frac{\pa_x
\psi}{v}\right)+F.
\end{split}
\end{equation}
Multiplying \eqref{pp.rw2} by
$\left[R\zeta+\left(\frac{1}{vv^{cd}\rho_{e}'(\phi^{cd})}-p^{cd}\right)\varphi\right]
vw,$ integrating the resulting equation over $\R_{+}$ leads to
\begin{equation}\label{c.eng2}
\begin{split}
\frac{1}{2}&\int_{\R_{+}}\left[R\zeta+\left(\frac{1}{vv^{cd}\rho_{e}'(\phi^{cd})}-p^{cd}\right)\varphi\right]^{2}(\partial_{x}\theta^{cd})^{2}dx\\
=&\frac{d}{dt}\int_{\R_{+}}\psi\left[R\zeta+\left(\frac{1}{vv^{cd}\rho_{e}'(\phi^{cd})}-p^{cd}\right)\varphi\right]vwdx\underbrace{-\int_{\R_{+}}
\psi\partial_{t}\left[R\zeta+\left(\frac{1}{vv^{cd}\rho_{e}'(\phi^{cd})}-p^{cd}\right)\varphi\right]vwdx}_{\CK_1}\\
&\underbrace{-\int_{\R_{+}}\psi\left[R\zeta+\left(\frac{1}{vv^{cd}\rho_{e}'(\phi^{cd})}-p^{cd}\right)\varphi\right]\partial_{t}vwdx}_{\CK_2}
\underbrace{-\int_{\R_{+}}\psi\left[R\zeta+\left(\frac{1}{vv^{cd}\rho_{e}'(\phi^{cd})}-p^{cd}\right)\varphi\right]v\partial_{t}wdx}_{\CK_3}\\
&\underbrace{-\int_{\R_{+}}\frac{\partial_{x}v}{v}\left[R\zeta+\left(\frac{1}{vv^{cd}\rho_{e}'(\phi^{cd})}-p^{cd}\right)\varphi\right]^{2}wdx}_{\CK_4}
\underbrace{+\mu\int_{\R_{+}}\frac{\partial_{x}\psi}{v}\partial_{x}
\left[\left(R\zeta+\left(\frac{1}{vv^{cd}\rho_{e}'(\phi^{cd})}-p^{cd}\right)\varphi\right)vw\right]dx}_{\CK_5}\\
&\underbrace{-\int_{\R_{+}}F\left[R\zeta+\left(\frac{1}{vv^{cd}\rho_{e}'(\phi^{cd})}-p^{cd}\right)\varphi\right]vwdx}_{\CK_6}
\underbrace{-\int_{\R_{+}}\frac{\partial_{x}\varphi}{v^{2}}
\sigma\left[R\zeta+\left(\frac{1}{vv^{cd}\rho_{e}'(\phi^{cd})}-p^{cd}\right)\varphi\right]vwdx}_{\CK_7}
\\&\underbrace{-\int_{\R_{+}}\frac{\partial_{x}v^{cd}}{v^{2}}
\sigma\left[R\zeta+\left(\frac{1}{vv^{cd}\rho_{e}'(\phi^{cd})}-p^{cd}\right)\varphi\right]vwdx}_{\CK_8}
\underbrace{+\int_{\R_{+}}\frac{\CM}{v}\partial_{x}\left[\left(R\zeta+\left(\frac{1}{vv^{cd}\rho_{e}'(\phi^{cd})}-p^{cd}\right)\varphi\right)vw\right]dx}_{\CK_9}
\\&\underbrace{+\int_{\R_{+}}\frac{\varphi\partial_{x}\phi^{cd}
}{vv^{cd}}\left[R\zeta+\left(\frac{1}{vv^{cd}\rho_{e}'(\phi^{cd})}-p^{cd}\right)\varphi\right]vwdx}_{\CK_{10}}.
%\\=&\partial_{t}\left(\int_{\R}\psi\left[R\zeta+\left(\frac{1}{vv^{cd}\rho_{e}'(\phi^{cd})}-p^{cd}\right)\varphi\right]vw\right)+\sum_{i=1}^{10}I_{i}.
\end{split}
\end{equation}
We now turn to compute $\CK_{l}$ $(1\leq l\leq 10)$ term by term. For the delicate term $\CK_1$, it can be rewritten as
\begin{equation}\label{ck1}
\begin{split}
\CK_{1} %=&-\int_{\R_{+}}\psi\partial_{t}\left[R\zeta+\left(\frac{1}{vv^{cd}\rho_{e}'(\phi^{cd})}-p^{cd}\right)\varphi\right]vwdx\\
=&-\int_{\R_{+}}\psi\partial_{t}\left(R\zeta-p^{cd}\varphi\right)vwdx-
\int_{\R_{+}}\psi\partial_{t}\left(\frac{1}{vv^{cd}\rho_{e}'(\phi^{cd})}\varphi\right)vwdx\\
=&-(\gamma-1)\int_{\R_{+}}\psi
vw\partial_{t}\left(\frac{R}{\gamma-1}\zeta+p^{cd}\varphi\right)dx+\gamma\int_{\R_{+}}\psi
vwp^{cd}\partial_{x}\psi dx\\&+\gamma\int_{\R_{+}}\psi
vw\partial_{t}p^{cd}\varphi dx-
\frac{1}{2}\int_{\R_{+}}\partial_{x}(\psi^{2})\left(\frac{w}{v^{cd}\rho_{e}'(\phi^{cd})}\right)dx-
\int_{\R_{+}}\psi\varphi\partial_{t}\left(\frac{1}{vv^{cd}\rho_{e}'(\phi^{cd})}\right)vwdx\\
=&\underbrace{(\gamma-1)\int_{\R_{+}}\psi
w\left(R\zeta-p^{cd}\varphi\right)\left(\partial_{x}u^{cd}+\partial_{x}\psi\right)dx}_{\CK_{1,1}}
\underbrace{+\kappa(\gamma-1)\int_{\R_{+}}\frac{v^{cd}\partial_{x}\zeta-\partial_{x}\theta^{cd}\varphi}{vv^{cd}}\partial_{x}\left(\psi
vw\right)dx}_{\CK_{1,2}}\\& \underbrace{-(\gamma-1)\int_{\R_{+}}\psi
vwGdx}_{\CK_{1,3}}\underbrace{-\frac{\gamma}{2}\int_{\R_{+}}p^{cd}\partial_{x}vw\psi^{2}dx}_{\CK_{1,4}}
\underbrace{-\frac{\gamma}{2}\int_{\R_{+}}\partial_{x}p^{cd}vw\psi^{2}dx}_{\CK_{1,5}}\underbrace{+\gamma\int_{\R_{+}}\psi
vw\partial_{t}p^{cd}\varphi dx}_{\CK_{1,6}}\\&+
\int_{\R_{+}}\left(\frac{1}{2v^{cd}\rho_{e}'(\phi^{cd})}-\frac{\gamma}{2}p^{cd}v\right)\psi^{2}(\partial_{x}\theta^{cd})^{2}dx
\underbrace{+\frac{1}{2}\int_{\R_{+}}\psi^{2}\partial_{x}\left(\frac{1}{v^{cd}\rho_{e}'(\phi^{cd})}\right)wdx}_{\CK_{1,7}}
\\&\underbrace{-\int_{\R_{+}}\psi\varphi\partial_{t}\left(\frac{1}{vv^{cd}\rho_{e}'(\phi^{cd})}\right)vwdx}_{\CK_{1,8}},
\end{split}
\end{equation}
where in the third identity we have used
\begin{equation*}\label{5.7}
\begin{split}
\frac{R}{\gamma-1}\pa_t\zeta+p^{cd}\pa_t\varphi=-\frac{R\zeta-p^{cd}\varphi}{v}\left(\partial_{x}u^{cd}+\partial_{x}\psi\right)
+\kappa\pa_x\left(\frac{v^{cd}\pa_x\zeta-\pa_x\theta^{cd}\varphi}{vv^{cd}}\right)+G,
\end{split}
\end{equation*}
which is derived from \eqref{pv} and \eqref{pta}.

Since $\rho_{e}'(\phi^{cd})<0$ according to the assumption $(\CA)_2$, \eqref{ck1} further implies
\begin{equation}\label{ps.ck1}
0<-\int_{\R_{+}}\left(\frac{1}{2v^{cd}\rho_{e}'(\phi^{cd})}-\frac{\gamma}{2}p^{cd}v\right)\psi^{2}(\partial_{x}\theta^{cd})^{2}dx
=-\CK_1+\sum\limits_{l=1}^{8}\CK_{1,l}.
\end{equation}
To compute $\CK_{1,l}$ $(1\leq l\leq 8)$ and $\CK_{l}$ $(2\leq l\leq 10)$, by applying \eqref{p.ine2}, \eqref{w.es},
Cauchy-Schwarz's inequality, Sobolev's inequality \eqref{sob.}, Young's inequality and  Lemmas \ref{pbd}
and \ref{es.tap}, we directly address the following
estimates:
\begin{equation*}\label{5.8}
\begin{split}
|\CK_{1,1}|+|\CK_{1,6}|+|\CK_{1,8}|+|\CK_{2}|\leq& C\int_{\R_{+}}|\psi
w|(|\pa_xu^{cd}|+|\pa_t\theta^{cd}|)(|\zeta|+|\varphi|)dx+C\int_{\R_{+}}|\psi
w\pa_x\psi|(|\zeta|+|\varphi|)dx
\\ \leq& C\|w\pa_t\theta^{cd}\|_{L^\infty}\|[\varphi,\psi,\zeta]\|^2
+C\de\|\pa_x\psi\|^{2}+\frac{C}{\de}\int_{\R_{+}}w^2\psi^2[\varphi,\zeta]^2dx
\\ \leq& C\delta\|\pa_x\psi\|^{2}+C\delta\eps_0(1+t)^{-\frac{3}{2}}+\frac{C}{\de^2}\|w\|^4_{L^\infty}\|\psi\|^2\|[\varphi,\zeta]\|^4
\\ \leq& C\delta\|\pa_x\psi\|^{2}+C\delta\eps_0(1+t)^{-\frac{3}{2}},
\end{split}
\end{equation*}
\begin{equation*}\label{5.9}
\begin{split}
|\CK_{1,2}|+|\CK_{5}|\leq&C\int_{\R_{+}}\left(|\partial_{x}\varphi|+|\partial_{x}\zeta|+|\partial_{x}\theta^{cd}\varphi|\right)
\left(|\partial_{x}\psi w|+\left|\psi(\partial_{x}\theta^{cd})^{2}\right|+|\psi w\partial_{x}v|\right)dx
\\ \leq&
C\delta\|\pa_x[\varphi,\psi,\zeta]\|^{2}+C\delta(1+t)^{-\frac{3}{2}},
\end{split}
\end{equation*}
\begin{equation*}\label{5.10}
\begin{split}
|\CK_{1,3}|\leq& C\int_{\R_{+}}\left|\psi
w\left((\partial_{x}u^{cd})^{2}+(\partial_{x}\psi)^{2}\right)\right|dx
\leq C\int_{\R_{+}}\psi^2
w^2(\partial_{x}u^{cd})^{2}dx+C\int_{\R_{+}}(\partial_{x}u^{cd})^{2}dx
+C
\delta\|\pa_x\psi\|^{2}
\\ \leq & C
\delta\|\pa_x\psi\|^{2}+C\delta(1+t)^{-\frac{3}{2}},
\end{split}
\end{equation*}
\begin{equation*}\label{5.11}
\begin{split}
|\CK_{1,4}|+|\CK_{1,5}|+|\CK_{1,7}|&\leq
C\int_{\R_{+}}\left|\partial_{x}\theta^{cd}w\psi^{2}\right|dx
+C\int_{\R_{+}}\left|\partial_{x}\varphi w\psi^{2}\right|dx
\\& \leq
C\delta\int_{\R_{+}}(\partial_{x}\theta^{cd})^{2}\psi^{2}dx+C\delta\|\pa_x[\psi,\varphi]\|^{2}+C\delta(1+t)^{-2},
\end{split}
\end{equation*}
%Plugging the above estimates into \eqref{ps.ck1}, one has
%\begin{equation}\label{ps.ck11}
%0<-\int_{\R_{+}}\left(\frac{1}{2v^{cd}\rho_{e}'(\phi^{cd})}-\frac{\gamma}{2}p^{cd}v\right)\psi^{2}(\partial_{x}\theta^{cd})^{2}dx
%\leq C\delta\|\pa_x[\varphi,\psi,\zeta]\|^{2}+C\delta(1+t)^{-\frac{3}{2}}
%+C\delta\eps_0\int_{\R_{+}}(\partial_{x}\theta^{cd})^{2}\psi^{2}dx.
%\end{equation}
%As to the estimates for $\CK_l$ with $2\leq l\leq 10$,
\begin{equation*}\label{5.12}
\begin{split}
|\CK_{3}|\leq
C\int_{\R_{+}}\left|\psi\left(|\zeta|+|\varphi|\right)\partial_{t}w\right|dx\leq
C\delta\eps_0(1+t)^{-\frac{3}{2}},
\end{split}
\end{equation*}
\begin{equation*}\label{5.11}
\begin{split}
|\CK_{4}|+|\CK_{10}|&\leq
C\int_{\R_{+}}\left|\partial_{x}\theta^{cd}w\left(\zeta^{2}+\varphi^{2}\right)\right|dx
+C\int_{\R_{+}}\left|\partial_{x}\varphi
w\left(\zeta^{2}+\varphi^{2}\right)\right|dx
\\ \leq& C\de(1+t)^{-1/2}\int_{\R_{+}}|\partial_{x}\theta^{cd}||x|\left(\|\pa_x\zeta\|^{2}+\|\pa_x\varphi\|^{2}+|\varphi(0,t)|^2\right)dx
\\&+C\de\|\pa_x\varphi,\zeta\|^{2}+\frac{C}{\de}\|w\|_{L^\infty}^2\left(\|\zeta\|\|\partial_{x}\zeta\|
+\|\varphi\|\|\partial_{x}\varphi\|\right)\|[\varphi,\zeta]\|^2
\\ \leq&
C\delta\|\varphi_{0}\|_{H^1}^{2}e^{-\frac{p_{-}}{\mu}t}+C\delta\|\pa_x[\varphi,\zeta]\|^{2}+C\delta(1+t)^{-2},
\end{split}
\end{equation*}
\begin{equation*}\label{5.13}
\begin{split}
|\CK_{6}|\leq&
C\delta^2(1+t)^{-\frac{1}{2}}\|F\|_{L^{1}}\left(\|\zeta\|_{\infty}+\|\varphi\|_{\infty}\right)
\leq
C\delta^2(1+t)^{-1}\left(\|\zeta\|^{\frac{1}{2}}\|\partial_{x}\zeta\|^{\frac{1}{2}}
+\|\varphi\|^{\frac{1}{2}}\|\partial_{x}\varphi\|^{\frac{1}{2}}\right)
\\ \leq& C\delta\|\pa_x[\varphi,\zeta]\|^{2}+C\delta(1+t)^{-\frac{4}{3}},
\end{split}
\end{equation*}
\begin{equation*}\label{5.14}
\begin{split}
|\CK_{7}|\leq&
C\int_{\R_{+}}\left|\partial_{x}\varphi\sigma(|\zeta|+|\varphi|)w\right|dx
\leq C\delta(1+t)^{-\frac{1}{2}}\left(\|\zeta\|_{\infty}+\|\varphi\|_{\infty}\right)\|\partial_{x}\varphi\|\|\sigma\|\\
\leq& C\delta(1+t)^{-\frac{1}{2}}\left(\|\zeta\|^{\frac{1}{2}}\|\partial_{x}\zeta\|^{\frac{1}{2}}
+\|\varphi\|^{\frac{1}{2}}\|\partial_{x}\varphi\|^{\frac{1}{2}}\right)\|\partial_{x}\varphi\|\|\sigma\|
\\ \leq& C\delta\|\partial_{x}[\varphi,\zeta]\|^{2}+C\delta(1+t)^{-2},
\end{split}
\end{equation*}
\begin{equation*}\label{5.15}
\begin{split}
|\CK_{8}|
\leq C\int_{\R_{+}}|\partial_{x}\theta^{cd}w||\si|\left(|\zeta|+|\varphi|\right)dx
\leq C\delta\|\varphi_{0}\|_{H^1}^{2}e^{-\frac{p_{-}}{\mu}t}
+C\delta\|\partial_{x}[\varphi,\zeta,\sigma]\|^{2},
\end{split}
\end{equation*}
%\begin{equation*}\label{5.15}
%\begin{split}
%\CK_{8}\leq&
%C\int_{\R_{+}}|\partial_{x}v^{cd}|\left|-\frac{\varphi}{vv^{cd}\rho_{e}'(\phi^{cd})}+M\right|\left(|\zeta|+|\varphi|\right)|w|dx\\
%\leq&C\int_{\R_{+}}|\partial_{x}\theta^{cd}w||\varphi|\left(|\zeta|+|\varphi|\right)dx
%+C\int_{\R_{+}}\left[\partial^{2}_{x}\sigma+\partial_{x}\sigma\partial_{x}v+\sigma^{2}
%+\partial_{x}v^{cd}\partial_{x}v\right](|\zeta|+|\varphi|)|\partial_{x}v^{cd}
%w|dx\\
%\leq&C\delta\|\varphi_{0}\|_{H^1}^{2}e^{-\frac{p_{-}}{\mu}t}
%+C\delta\|\partial_{x}[\varphi,\zeta,\sigma,\partial_{x}\sigma]\|^{2}+C\delta(1+t)^{-\frac{4}{3}},
%\end{split}
%\end{equation*}
\begin{equation*}\label{5.16}
\begin{split}
|\CK_{9}|\leq&
C\int_{\R_{+}}\left[|\partial^{2}_{x}\sigma|+|\partial_{x}\sigma\partial_{x}v|+|\sigma^{2}|
+|\partial_{x}v^{cd}\partial_{x}v|\right]\left[(|\partial_{x}\zeta|+|\partial_{x}\varphi|)|w|+(|\zeta|+|\varphi|)
((\partial_{x}\theta^{cd})^{2}+|\partial_{x}v
w|)\right]dx\\
\leq&C\delta\|\partial_{x}[\varphi,\zeta,\sigma,\partial_{x}\sigma]\|^{2}+C\delta(1+t)^{-\frac{3}{2}}.
\end{split}
\end{equation*}
We now plug the above estimates for $\CK_{1,l}$ with $1\leq l\leq 8$ into \eqref{ps.ck1} to obtain
\begin{equation}\label{ps.ck11}
\begin{split}
&\left|\CK_1-\int_{\R_{+}}\left(\frac{1}{2v^{cd}\rho_{e}'(\phi^{cd})}-\frac{\gamma}{2}p^{cd}v\right)\psi^{2}(\partial_{x}\theta^{cd})^{2}dx\right|
\\& \qquad\leq C\delta\|\pa_x[\varphi,\psi,\zeta]\|^{2}+C\delta(1+t)^{-\frac{3}{2}}
+C\delta\int_{\R_{+}}(\partial_{x}\theta^{cd})^{2}\psi^{2}dx.
\end{split}
\end{equation}
Next by substituting the estimates for $\CK_{l}$ $(2\leq l\leq 10)$
and \eqref{ps.ck11} into \eqref{c.eng2} and integrating the
resulting equality with respect to time over $[0,T]$, one has
\begin{equation}\label{c.eng3}
\begin{split}
-\int_{\R_{+}}&\psi\left[R\zeta+\left(\frac{1}{vv^{cd}\rho_{e}'(\phi^{cd})}-p^{cd}\right)\varphi\right]vwdx
\\&+\int_{0}^{T}\int_{\R_{+}}\left\{\frac{1}{2}\left[R\zeta+\left(\frac{1}{vv^{cd}\rho_{e}'(\phi^{cd})}-p^{cd}\right)\varphi\right]^{2}
+\left(-\frac{1}{2v^{cd}\rho_{e}'(\phi^{cd})}+\frac{\gamma}{2}p^{cd}v-C\delta\right)\psi^{2}\right\}(\partial_{x}\theta^{cd})^{2}dxdt\\
\leq&C\delta+C\de\eps_0+C\delta\left\|\varphi_{0}\right\|^{2}_{H^1}+
C\delta\int_{0}^{T}\|\partial_{x}[\varphi,\psi,\zeta,\sigma,\partial_{x}\sigma\|^{2}dt,
\end{split}
\end{equation}
for suitably small $\de>0$ and $\eps_0>0$.

Let us now define
\begin{equation}\label{fe.def}
\CE(\varphi,\psi,\zeta)=-\int_{\R_{+}}\psi\left[R\zeta+\left(\frac{1}{vv^{cd}\rho_{e}'(\phi^{cd})}-p^{cd}\right)\varphi\right]vwdx,
\end{equation}
then \eqref{c.eng} follows from \eqref{c.eng3} and \eqref{fe.def}, this ends the proof of Lemma \ref{key.es}.

\end{proof}

Finally we give the detailed proof of \eqref{ptsi}.

\begin{proof}[Proof of \eqref{ptsi}] %To prove \eqref{ptsi}, we first get from \eqref{pp} and \eqref{aps} that
%\begin{equation}\label{hes.si}
%\|\pa_x^2\sigma\|\leq C\eps_0.
%\end{equation}
Taking the inner product of $\pa_t\eqref{pp}$ with $\pa_t\si$ with respect to $x$ over $\R_+$, one
has
\begin{equation}\label{ptsi.es}
\begin{split}
&\underbrace{\int_{\R_{+}}\pa_t\left(\frac{v^{cd}}{v}\right)\pa^2_x\sigma\pa_t\sigma
dx}_{\CJ_1}\underbrace{+
\int_{\R_{+}}\frac{v^{cd}}{v}\pa_t\pa^2_x\sigma\pa_t\sigma dx}_{\CJ_2}
\underbrace{- \int_{\R_{+}}\pa_t\left(\frac{v^{cd}}{v^2}\pa_xv\right)\pa_x\sigma
\pa_t\sigma dx}_{\CJ_3}\underbrace{-\int_{\R_+}
\frac{v^{cd}}{v^2}\pa_xv\pa_t\pa_x\sigma\pa_t\sigma dx}_{\CJ_4}
\\&\quad=-\int_{\R_{+}}\pa_t\varphi\pa_t\sigma
dx+\int_{\R_{+}}\pa_tv\left(1-v^{cd}\rho_{e}(\sigma+\phi^{cd})\right)\pa_t\sigma
dx
\\&\qquad+\int_{\R_{+}}v\pa_t\left(1-v^{cd}\rho_{e}(\sigma+\phi^{cd})\right)\pa_t\sigma dx
-\int_{\R_{+}}\pa_t\left(v^{cd}\pa_x\left(\frac{\pa_x\phi^{cd}}{v}\right)\right)\pa_t\sigma
dx.
\end{split}
\end{equation}
We turn our attention first to $\CJ_l$ $(1\leq l\leq4)$ which can not be directly controlled.
Since $\si(0,t)=\si(+\infty,t)=0$, by integration by parts and using the cancellation, we find
\begin{equation}\label{CJ24}
\begin{split}
\CJ_2+\CJ_4=-\int_{\R_{+}}\frac{v^{cd}}{v}\pa_t\pa_x\sigma\pa_t\pa_x\sigma dx-\int_{\R_{+}}\frac{\pa_xv^{cd}}{v}\pa_t\pa_x\sigma\pa_t\sigma dx,
\end{split}
\end{equation}
and
\begin{equation}\label{CJ13}
\begin{split}
\CJ_1+\CJ_3=\int_{\R_{+}}\frac{\pa_tv^{cd}}{v}\pa^2_x\sigma\pa_t\sigma dx
+\int_{\R_{+}}\frac{v^{cd}}{v^2}\pa_tv\pa_x\sigma\pa_x\pa_t\sigma dx
-\int_{\R_{+}}\frac{\pa_tv^{cd}}{v^2}\pa_xv\pa_x\sigma
\pa_t\sigma dx.
\end{split}
\end{equation}
On the other hand, similar to \eqref{taylor}, one has
\begin{equation}\label{taylor2}
1-v^{cd}\rho_{e}(\sigma+\phi^{cd})=-v^{cd}\rho_{e}'(\phi^{cd})\sigma
\underbrace{-v^{cd}\int_{\phi^{cd}}^{\phi}\rho_{e}''(\varrho)(\phi-\varrho)d\varrho}_{\CJ_0},
\end{equation}
and moreover
\begin{equation}\label{ptR2}
\pa_t\CJ_0\thicksim \pa_t\si\si+\pa_tv^{cd}\si+\pa_tv^{cd}\si^2.
\end{equation}
Substituting \eqref{CJ24}, \eqref{CJ13}, \eqref{taylor2} and \eqref{ptR2} into \eqref{ptsi.es} and applying \eqref{pv} and \eqref{NSPl}, we deduce
\begin{equation*}\label{ptsi.es2}
\begin{split}
\int_{\R_{+}}&\frac{v^{cd}}{v}|\pa_t\pa_x\sigma|^2 dx-\int_{\R_{+}}vv^{cd}\rho_{e}'(\phi^{cd})|\pa_t\sigma|^2dx\\
\leq& \left|\int_{\R_{+}}\frac{\pa_xv^{cd}}{v}\pa_t\pa_x\sigma\pa_t\sigma dx\right|+\left|\int_{\R_{+}}\frac{\pa_tv^{cd}}{v}\pa^2_x\sigma\pa_t\sigma dx\right|+\left|\int_{\R_{+}}\frac{v^{cd}}{v^2}\pa_tv\pa_x\sigma\pa_x\pa_t\sigma dx\right|
\\&+\left|\int_{\R_{+}}\frac{\pa_tv^{cd}}{v^2}\pa_xv\pa_x\sigma
\pa_t\sigma dx\right|+C\int_{\R_{+}}|\pa_x\psi\pa_t\sigma|dx+C\int_{\R_{+}}|\pa_xu\si\pa_t\sigma|dx
\\&+C\int_{\R_{+}}|\pa_tv^{cd}\si\pa_t\sigma|dx+C\int_{\R_{+}}|\pa_t\si\si\pa_t\sigma|dx+C\int_{\R_{+}}|\pa_tv^{cd}\si^2\pa_t\sigma|dx
\\&+C\int_{\R_{+}}|\pa_t\pa_x^2v^{cd}\pa_t\sigma|dx+C\int_{\R_{+}}|\pa_tv^{cd}\pa^2_xv^{cd}\pa_t\sigma|dx
+C\int_{\R_{+}}|\pa_tv^{cd}\pa_xv^{cd}\pa_xv\pa_t\sigma|dx
\\&+C\int_{\R_{+}}|\pa_xv^{cd}\pa_x^2u\pa_t\sigma|dx,
\end{split}
\end{equation*}
which yields \eqref{ptsi}, according to Cauchy-Schwarz's inequality,
\eqref{p.ine2} and Lemma \ref{es.tap}. This completes the proof of \eqref{ptsi}.

\end{proof}

\medskip
\noindent {\bf Acknowledgements:} The first author was supported by
grants from the National Natural Science Foundation of China
\#11471142 and \#11271160. The second and third authors were
supported by the National Natural Science Foundation of China
\#11331005, the Program for Changjiang Scholars and Innovative
Research Team in University \#IRT13066, the Scientific Research
Funds of Huaqiao University (Grant No.15BS201), and the Special Fund
Basic Scientific Research of Central Colleges \#CCNU12C01001. The
first and second authors would like to thank Professor Renjun Duan
for many fruitful discussions on the topic of the paper.

\end{document}